\documentclass[12pt]{amsart}
\usepackage{graphicx} 
\usepackage{amsmath,amsfonts,amsthm,amssymb}
\usepackage[all]{xy}
\usepackage{comment}
\usepackage{dsfont}
\usepackage{xcolor}
\usepackage{tabularx}
\usepackage{hyperref}
\usepackage{multirow}
\usepackage{ytableau}
\usepackage{tikz-cd}
\usepackage[
backend=biber,
style=alphabetic,
sorting=nty,
giveninits=true
]{biblatex}
\usepackage[margin=1in]{geometry}
\newtheorem{theorem}{Theorem}
\newtheorem{definition}[theorem]{Definition}
\newtheorem{proposition}[theorem]{Proposition}
\newtheorem{lemma}[theorem]{Lemma}
\newtheorem{corollary}[theorem]{Corollary}

\theoremstyle{remark}
\newtheorem{example}[theorem]{Example}
\newtheorem{remark}[theorem]{Remark}
\newtheorem{hypo}[theorem]{Hypothesis}
\numberwithin{theorem}{section}

\title{Polynomial functors in $\Ver_4^+$}
\author{Kevin Coulembier}
\author{Serina Hu}

\newcommand{\mfrak}{\mathfrak}
\newcommand{\mbb}{\mathbb}
\renewcommand{\1}{\mathds{1}}

\DeclareMathOperator{\Ver}{Ver}
\DeclareMathOperator{\Span}{Span}

\DeclareMathOperator{\End}{End}
\DeclareMathOperator{\uEnd}{\underline{End}}
\DeclareMathOperator{\Hom}{Hom}
\DeclareMathOperator{\uHom}{\underline{Hom}}
\DeclareMathOperator{\im}{im}

\DeclareMathOperator{\Ind}{Ind}
\DeclareMathOperator{\Grp}{Grp}

\DeclareMathOperator{\GL}{GL}
\DeclareMathOperator{\Sym}{Sym}

\DeclareMathOperator{\Rep}{Rep}
\DeclareMathOperator{\Comod}{Comod}

\DeclareMathOperator{\sVec}{sVec}
\DeclareMathOperator{\Vecc}{Vec}
\DeclareMathOperator{\Fun}{Fun}

\DeclareMathOperator{\Aut}{Aut}
\DeclareMathOperator{\Mat}{Mat}

\newcommand{\assign}{:=}

\newcommand{\mN}{\mathbb{N}}
\newcommand{\TC}{\mathbf{TC}}
\newcommand{\mA}{\mathbf{A}}
\newcommand{\bA}{\mathbf{A}}
\newcommand{\bB}{\mathbf{B}}
\newcommand{\bk}{k} 
\newcommand{\mZ}{\mathbb{Z}}
\DeclareMathOperator{\Res}{Res}
\newcommand{\unit}{\1}
\newcommand{\Young}{\mathbf{Young}}
\newcommand{\SYoung}{\mathbf{SYoung}}
\newcommand{\UnEn}{\mathbf{UnEn}}
\newcommand{\cC}{\mathcal{C}}
\newcommand{\mP}{\mathbb{P}}
\newcommand{\cD}{\mathcal{D}}
\newcommand{\Pol}{\mathrm{Pol}}
\newcommand{\SPol}{\mathrm{SPol}}
\newcommand{\mD}{\mathfrak{D}}

\newcommand{\mF}{\mathfrak{F}}
\newcommand{\mG}{\mathfrak{G}}
\newcommand{\bT}{\mathbf{T}}
\DeclareMathOperator{\Indec}{Idc}
\newcommand{\tto}{\twoheadrightarrow}
\DeclareMathOperator{\Sds}{Sds}
\newcommand{\Fr}{\mathrm{Fr}}
\newcommand{\Frob}{\mathrm{Frob}}
\newcommand{\CAlg}{\mathsf{CAlg}}
\usepackage{breqn}
\usepackage{makecell}

\begin{document}

\maketitle

\begin{abstract}
    We study polynomial functors in the incompressible category $\Ver_4^+$, which can be viewed as super polynomial functors in characteristic 2. Concretely, we classify additive, exact and simple polynomial functors, and describe how simple polynomial functors evaluate on arbitrary objects. We also determine which objects are not annihilated by any polynomial functors of a given degree and for which objects the symmetric group algebra acts faithfully via the braiding.
\end{abstract}

\tableofcontents

\section{Introduction}

Polynomial functors were introduced in \cite{friedlander_cohomology_1997} as a tool to prove finite generation of the cohomology algebras of finite group schemes over fields (of positive characteristic). They can be interpreted as polynomial representations of $\GL_n$ for $n\gg0$ and have since appeared extensively in representation and category theory. In \cite{Axtell} `super polynomial functors' were introduced, and they have similarly led to a proof of finite generation of the cohomology algebras of finite group superschemes in \cite{drupieski_cohomology_2016}. We can think of (super) polynomial functors as polynomial functors associated to the symmetric tensor category, in the sense of \cite{EGNO}, of (super) vector spaces.

Over fields of characteristic zero, the main structure theory of (symmetric) tensor categories was established by Deligne, see \cite{deligne_categories_1990, deligne_tensorielles_2002} and any theory of polynomial functors is immediately seen to be universal and must correspond to Schur functors. 
In positive characteristic, the structure theory of tensor categories is much more involved and not completely understood, see \cite{benson_new_2021, coulembier_frobenius_2022} and references therein. Several equivalent approaches to polynomial functors for tensor categories were developed in \cite{coulembier_inductive_2024}. Here, the polynomial functors depend heavily on the choice of tensor category, and the choices of (super) vector spaces recover the familiar notions discussed above. By results in \cite{Incomp}, it becomes a natural important goal to understand polynomial functors in all `incompressible' tensor categories. This is in particular motivated by the conjecture of finite generation of cohomology of finite tensor categories in \cite{EO-finite}, and the search for higher Frobenius functors as in \cite{CF}. The complete list of currently known incompressible categories has been obtained in \cite{benson_new_2021, coulembier_monoidal_2021}. The first natural examples beyond (super) vector spaces are $\Ver_4^+$ in characteristic $2$ and $\Ver_p$ in characteristic $p\ge 5$. In the current paper, we focus on the former case.

To describe effectively the category of polynomial functors for $\Ver_4^+$, in Section~\ref{sec-polyfun} we revisit the theory of inductive systems and polynomial functors from \cite{coulembier_inductive_2024}, and obtain some new results on generators of inductive systems and on additive and exact polynomial functors. We then apply that theory to complete some existing results from \cite{Axtell, Giordano} on super polynomial functors, or provide new proofs.

Our main results, obtained in Section~\ref{sec-Ver4}, are concerned with polynomial functors for $\Ver_4^+$. We classify simple polynomial functors, by combining the classification of simple representations of general linear groups in $\Ver_4^+$ from \cite{hu_supergroups_2024} with techniques of odd reflections inspired by supergroups, used similarly in \cite{Brundan-Kujawa}. We then classify all additive (not necessarily semisimple) and exact polynomial functors. Finally, we determine which objects in $\Ver_4^+$ have the property that no polynomial functors of a given degree vanish on them, by describing an analogue of the Mullineux involution in characteristic 2. We conclude by some results on invariant theory.

\section{Preliminaries}\label{sec-prel}
Unless further specified, we let $k$ be an algebraically closed field.

\subsection{Tensor categories}\label{sec-intro-tenscat}
We refer to \cite{EGNO} for all claims and unexplained concepts in this subsection. A $k$-linear monoidal category $(\cC,\otimes,\1)$ is a tensor category if
\begin{enumerate}
    \item $\cC$ is abelian, with all objects of finite length;
    \item $k\to\End(\1)$ is an isomorphism;
    \item $(\cC,\otimes,\1)$ is rigid in the sense that every object admits monoidal duals.
\end{enumerate}
A tensor functor is an exact $k$-linear monoidal functor. It follows that $\unit$ is simple and that tensor functors are faithful. Unless specifically mentioned otherwise we will henceforth take the convention that tensor category and functor automatically refer to symmetric tensor categories and functors.

The most basic examples are the category of finite dimensional vector spaces $\Vecc$ and the category of finite dimensional supervector spaces $\sVec$ (if $\mathrm{char}(k)\not=2$), which are $\mZ/2$-graded vector spaces with sign twisted symmetric braiding. The two simple objects are then the unit $\1$ and the odd line $\bar{\1}$.

Of particular importance for the study of polynomial functors are incompressible tensor categories, such as $\Vecc$ and $\sVec$ and $\Ver_4^+$ as in \S \ref{ver4_intro} below, see \cite{Incomp} for an overview.

We denote the category of commutative algebras in the ind-completion of a tensor category $\cC$ by 
$\mathsf{CAlg}\cC$. An affine group scheme in $\cC$ is a representable functor $\mathsf{CAlg}\cC\to\Grp$, where $\Grp$ is the category of groups. It is thus represented by a commutative Hopf algebra. The most important example for us is the general linear group $\GL(X)$ of an object $X$, which sends $A\in\mathsf{CAlg}\cC$ to the automorphism group of the $A$-module $A\otimes X$. For example $\GL(\1)$ is the ordinary multiplicative group $\mathbb{G}_m$ interpreted as an affine group scheme in $\cC$ represented by the Hopf algebra $k[x,x^{-1}]$ over $k$ via $\Vecc\subset\cC$.

We have two important group homomorphisms into $\GL(X)$. The first is $\mathbb{G}_m\to\GL(X)$ that sends an automorphism $A\to A$ to its tensor product with $X$. This turns every $\GL(X)$-representation in $\cC$ into a $\mathbb{G}_m$-representation, {\it i.e.} a $\mZ$-graded object in $\cC$. In particular, a simple $\GL(X)$-representation $V$ yields an object contained in one degree, called the degree of $V$. The second is 
$\varepsilon:\pi_1(\cC)\to \GL(X)$, where $\pi_1(\cC)$ is the fundamental group of $\cC$ from \cite{deligne_categories_1990}. The group scheme $\pi_1(\cC)$ sends $A\in\mathsf{CAlg}\cC$ to the automorphism group of $A\otimes-$ as a monoidal functor from $\cC$ to the category of $A$-modules. The morphism $\varepsilon$ is simply evaluation at $X$. The category $\Rep(\GL(X),\varepsilon)$ is the tensor category of $\GL(X)$-representations in $\cC$ such that the action of $\pi_1(\cC)$ obtained via $\varepsilon$ gives the natural action coming from the definition of $\pi_1(\cC)$. A $\GL(X)$-representation is polynomial of degree $d$ if it occurs as a subquotient of $X^{\otimes d}\otimes Y$ for some $Y\in\cC$. We consider the corresponding abelian category $\Rep^d\GL(X)$ and its subcategory $\Rep^d(\GL(X),\varepsilon)$.

The Deligne tensor product of two tensor categories $\cC_1,\cC_2$ (or more general $k$-linear abelian categories with similar finiteness conditions) will be denoted by $\cC_1\boxtimes\cC_2$.

By a pseudo-abelian cateogry, we mean an additive category that is idempotent complete (Karoubi). For two pseudo-abelian categories $\bA$, $\bB$ we denote by $\bA\otimes_k\bB$ the Karoubi completion of their ordinary tensor product as $k$-linear categories (their coproduct).
For example, for finite groups $G,H$, we can realise
$$(\Rep G)\otimes_k (\Rep G)\;\subset\; \Rep(G\times H)\;\simeq\;\Rep G \boxtimes \Rep H$$
as the category of direct summands of exterior tensor products $V_1\boxtimes V_2$.

\subsection{The higher Verlinde category $\Ver_4^+$}\label{ver4_intro}

Let $k$ be an algebraically closed field of characteristic 2, not necessarily algebraically closed. Consider the Hopf algebra $k[\delta]/\delta^2$ where $\delta$ is primitive. The higher Verlinde category $\Ver_4^+=\Ver_4^+(k)$ has a concrete description as the monoidal category of finite-dimensional $k[\delta]/\delta^2$-modules with symmetric braiding given by the R-matrix $1 \otimes 1 + \delta \otimes \delta$, i.e.
\begin{equation*}
   c=c_{V,W}:V\otimes W\xrightarrow{\sim}W\otimes V,\quad c(v \otimes w) = w \otimes v + \delta w \otimes \delta v.
\end{equation*}
From now on, we abbreviate $\delta x$ to $x'$.

The category $\Ver_4^+$ has one simple object $\1$, the trivial representation, with indecomposable self-dual projective cover $P=k[\delta]/\delta^2$ the regular representation. These are the only indecomposable representations, so every object in $\Ver_4^+(k)$ is of the form $m\1 + nP$. For brevity, we will often denote $m\1$ by $m$.

Notice that $\Ver_4^+$ admits a forgetful monoidal functor to $\text{Vec}_k$, although this functor does not respect the braiding. Nevertheless, this means we can treat objects in $\Ver_4^+$ concretely as vector spaces with additional structure (an endomorphism $\delta$ such that $\delta^2 = 0$), and any relations involving the braiding can be written in terms of elements of these vector spaces and the $\delta$-action using the $R$-matrix. Hence, we can speak of elements and bases of objects in $\Ver_4^+$.

For example, in $\Ver_4^+$, a commutative algebra $A$ corresponds to an ordinary, possibly non-commutative algebra over $k$ equipped with a derivation $\delta$ with $\delta^2 = 0$. Since $\delta$ is primitive, 
\begin{equation*}
    (ab)' = a'b + ab'
\end{equation*}
and the commutativity of $A$ in $\Ver_4^+(k)$ is equivalent to
\begin{equation*}
    ab = ba + a'b'.
\end{equation*}
In particular, we obtain that
\begin{equation*}
    (a^n)' = n a^{n-1}a', \qquad
    (a')^2 = 0,
\end{equation*}
and that $\ker \delta$ is in the center of $A$.

We denote by $H$ the $k$-algebra $k[t]/(t^2)$ interpreted as a (commutative) algebra in $\Ver_4^+$, {\it i.e.} $t'=0$. It can be made into a cocommutative Hopf algebra with $t$ primitive. As a functor of points, it sends $A$ to the additive group of elements that square to 0 in $\ker \delta \subset A$. As proved in \cite{hu_supergroups_2024}, we can describe the fundamental group of $\Ver_4^+$:

\begin{proposition}
    The fundamental group $\pi_1$ of $\Ver_4^+$ has coordinate algebra $k[\pi_1] = H.$ An element $b \in \pi_1(A)\subset A$ acts as the identity on $A \otimes \1$ and as $\begin{pmatrix} 1 & 0 \\ b & 1 \end{pmatrix}$ on $A \otimes P$.
\end{proposition}

\subsection{General linear groups in $\Ver_4^+$}\label{GLX:intro}

The irreducible representations of general linear groups in $\Ver_4^+$ were studied in \cite{hu_supergroups_2024}. Here we review some results that will be useful for \S \ref{sec-Ver4}. Let $X = m \1 + nP$. For a commutative algebra $A \in \CAlg(\Ver_4^+)$, let $\Mat_{i, j}(A)$ denote the $i \times j$ matrices with entries in $A$. For $M \in \Mat_{i, j}(A)$, let $M' \in \Mat_{ij}(A)$ be the matrix we get from applying $\delta$ to each entry of $M$, i.e. $M'_{\ell_1, \ell_2} = (M_{\ell_1, \ell_2})'$. 

 As a functor of points, $\GL(X)$ sends a commutative algebra $A \in \CAlg(\Ver_4^+)$ to $\Aut_A(X \otimes A)$, which we can identify with the group of $(m + 2n) \times (m + 2n)$ matrices with entries in $A$ of the form \begin{equation}\label{eq:matrix} \begin{pmatrix} F & C & C' \\ B' & D & D' \\ B & E & D + E' \end{pmatrix}\end{equation} where $F$ is an invertible $m \times m$ matrix with entries in $\ker \delta \subset A$ (i.e. $F \in \GL(m)(A)$), $D, E \in \Mat_{n, n}(A)$, $C \in \Mat_{m ,n }(A)$, $B \in \Mat_{n, m}(A)$, and $D$ is invertible. For example, if the $i$-the basis vector of $m\unit$ is denoted by $e_i$ and the generator of the $j$-th copy of P denoted by $f_j$, then the element corresponding to the above matrix sends $e_i$ to $\sum_{m}F_{mi}\otimes e_m+\sum_n(B_{ni}\otimes f_n)'$. Let $\det\in A$ be the determinant of \eqref{eq:matrix}, which is well-defined modulo the image of $\delta$, since changing the order of the variables changes $\det$ by an element in the ideal generated by the image of $\delta$.

\begin{proposition}
    The determinant of a matrix in $\Mat_{m+2n}(A)$ as in \eqref{eq:matrix} can be represented by an element of $\ker \delta\subset A$.
\end{proposition}
\begin{proof}
    As described in \cite{hu_supergroups_2024} Section 3, a matrix in $\GL(nP)(A)$ of the form $\begin{pmatrix} D & D' \\ E & D + E'\end{pmatrix}$ can be factored as
    \begin{equation*}
        \begin{pmatrix} I_n & 0 \\ ED^{-1} & I_n + (ED^{-1})' \end{pmatrix} \begin{pmatrix} D & D' \\ 0 & D \end{pmatrix}.
    \end{equation*}
    Hence, its determinant is $\det(I_n + (ED^{-1})')  \cdot (\det D)^2$, which lies in the kernel of $\delta$ as the entries of $I_n + (ED^{-1})'$ all lie in $\ker \delta$ and $(\det D)^2 \in \ker \delta$.

    An arbitrary matrix in $\GL(m + nP)(A)$ can also be factored as
    \begin{equation*}
    \begin{pmatrix} F & C & C' \\ B' & D & D' \\ B & E & D + E' \end{pmatrix}
    = \begin{pmatrix}
        I_m & 0 & 0 \\
        B'F^{-1} & I_n & 0 \\
        BF^{-1} & 0 & I_n
    \end{pmatrix}
    \begin{pmatrix}
        F & C & C' \\
        0 & B'F^{-1}C + D & (B'F^{-1}C + D)' \\
        0 & BF^{-1}C + E & B'F^{-1}C + D + (BF^{-1}C + E)'
    \end{pmatrix}.
    \end{equation*}
    Therefore, the determinant of this matrix is the determinant of the right matrix here. Notice that $F$ has entries in $\ker \delta$ and that $M = \begin{pmatrix} B'F^{-1}C + D & (B'F^{-1}C + D)' \\
        BF^{-1}C + E & B'F^{-1}C + D + (BF^{-1}C + E)' \end{pmatrix}$
    is a matrix in $\GL(nP)(A)$, so its determinant is in $\ker \delta$. Hence, $(\det F)(\det M)$ is also in the kernel of $\delta$.
\end{proof}

\begin{remark}
    Note that in general $\det$ will also have representatives not in $\ker\delta$. For every monomial in $\delta$, we must choose an ordering for the variables, and different orderings may yield representatives not in $\ker \delta$. For instance, let $x, y, x', y'$ be a basis of $2P$. Then for a commutative algebra $A$, if we pick $a, b, c, d \in A$ such that $ad - bc \in A^\times$, the matrix
    \begin{equation*}
        \begin{pmatrix} a & b & a' & b' \\
        c & d & c' & d' \\
        0 & 0 & a & b \\
        0 & 0 & c & d
        \end{pmatrix}
    \end{equation*}
    is an element of $\GL(2P)(A)$. One representative of the determinant of this matrix is $(ad+bc)^2$, which is in the kernel of $\delta$. Another representative is $(ad + bc)(da + cb)$. Since the difference $a'd'(ad+bc)$ between the two representatives is not in the kernel, the second representative is not in $\ker\delta$.

    
\end{remark}

\color{black}

    As an algebra, the coordinate ring $k[\GL(X)]$ is the commutative algebra (in $\Ver_4^+$) freely generated by $F_{ij}$, $B_{aj}, B'_{aj}$, $C_{ib}, C'_{ib}$, $D_{ab}, D'_{ab}, E_{ab}, E'_{ab}$, with $1\le i,j\le m$ and $1\le a,b\le n$, and $u$, modulo $u \cdot \det = \det \cdot u = 1$. Here $\det\in\ker\delta$ is an arbitrarily chosen representative of the determinant. As discussed in \cite{benson2025groupschemesliealgebras}, the effect of adjoining the inverse of $\det$ is independent of the order of the variables we choose for $\det$.

    The coproduct is given by
    \begin{align*}
        \Delta (F_{ij}) &= \sum_{k = 1}^m F_{ik} \otimes F_{kj} + \sum_{k=1}^n (C_{ik} \otimes B'_{kj} + C'_{ik} \otimes B_{kj}) \\
        \Delta(B_{ij}) &= \sum_{k = 1}^m B_{ik} \otimes F_{kj} + \sum_{k=1}^n (E_{ik} \otimes B'_{kj} + (D_{ik} + E'_{ik}) \otimes B_{kj}) \\
        \Delta(C_{ij}) &= \sum_{k = 1}^m F_{ik} \otimes C_{kj} + \sum_{k = 1}^n (C_{ik} \otimes D_{kj} + C'_{ik} \otimes E_{kj}) \\
        \Delta(D_{ij}) &= \sum_{k=1}^m B'_{ik} \otimes C_{kj} + \sum_{k = 1}^n (D_{ik} \otimes D_{kj} + D'_{ik} \otimes E_{kj}) \\
        \Delta(E_{ij}) &= \sum_{k=1}^m B_{ik} \otimes C_{kj} + \sum_{k= 1}^n (E_{ik} \otimes D_{kj} + (D_{ik} + E'_{ik}) \otimes E_{kj}).
    \end{align*}

\begin{example}
    Elements of $\GL(P)(A)$ are of the form $(a, b) = \begin{pmatrix} a & a' \\ b & a + b' \end{pmatrix}$ with $a \in A^\times$.
\end{example}

\begin{remark}\label{rem:epsilon}
The natural map $\varepsilon: \pi_1 \to \GL(X)$, $X = m + nP$ satisfies $\varepsilon(b) = \begin{pmatrix} I_m & 0 & 0 \\ 0 & I_n & 0 \\ 0 & bI_n & I_n \end{pmatrix}$ for $b \in \pi_1(A)$. Likewise, the map $k[\GL(X)] \mapsto k[\pi_1]$ sends $F_{ii}, D_{ii} \mapsto 1$, $E_{ii} \mapsto t$, and all other degree 1 generators to 0. 
\end{remark}

The following definition will be discussed in more detail and generality in Section~\ref{sec:additive}.
\begin{definition}
    Denote by $k[\End(X)] \subset k[\GL(X)]$ the polynomial ring in $F_{ij}, 1 \le i, j \le m$, $B_{ij}, B'_{ij}, 1 \le i \le n, 1 \le j \le m$, $C_{ij}, C'_{ij}$, $1 \le i \le m$, $1 \le j \le n$, $D_{ij}, D'_{ij}, E_{ij}, E'_{ij}, 1 \le i, j \le n$, with coproduct $\Delta_\times$ from restriction of the coproduct $\Delta$ on $k[\GL(X)]$. As an algebra, it is isomorphic to the coordinate ring of $\End(X)$, the group sending $A \mapsto \End_A(A \otimes X)$, and hence is equipped with another coproduct $\Delta_+$.
\end{definition}

From now on, when we refer to a representation $M$ of $\GL(X)$ in $\Ver_4^+$, we assume that $M \in \Ver_4^+$, not in the ind-category, and the $\GL(X)$-action is compatible with the action of $\pi_1$ on $M$. We thus consider $\Rep(\GL(X), \varepsilon)$.

We review the irreducible representations of $\GL(P)$ per degree, as in \S \ref{sec-intro-tenscat}. Let $T_n = S^n(P)$ for $n = 1, 2, 3$, let $\chi$ be the $\GL(P)$-character where $(a, b)$ acts as $a^4$, and $\xi$ be the $\GL(P)$-character where $(a, b)$ acts as $1 + (ba^{-1})'$. Then in degree $\ell = 4q + r$ with $0 \le r\le 3$, the irreducible representations of $\GL(P)$ are $$ \begin{cases} \chi^q, \chi^q \otimes \xi & r = 0, \\ T_r \otimes \chi^q, T_r \otimes \chi^q \otimes \xi & r = 1, 3, \\ T_r  \otimes \chi^q & r = 2. \end{cases}$$

The irreducible representations of $\GL(m + nP)$ can, via a choice of Borel subgroup, be labelled by their `highest weights'. The latter are by definition simple representations of $\mathbb{G}_m^{\times m}\times \GL(P)^{\times n}$. Concretely, the irreducible representations of $\GL(m + nP)$ correspond to tuples $$(\lambda|\Lambda) = (\lambda_1, \dots, \lambda_m|\Lambda_1, \dots, \Lambda_n)$$ such that $\lambda_i \in \mbb{Z}$, $\lambda_1 \ge \cdots \ge \lambda_m$, and $\Lambda_i \in \text{Irrep}(\GL(P))$, $\deg \Lambda_1 \ge \cdots \ge \deg \Lambda_n$. In particular, irreducible representations of $\GL(m + nP)$ correspond to pairs of $\GL(m)$-irreps and $\GL(nP)$-irreps.

\section{Generalities on polynomial functors} \label{sec-polyfun}
Let $k$ be a field. We review and expand some notions from \cite{coulembier_inductive_2024}.

\subsection{Closed inductive systems}

We denote by $S_n$ the permutation group of the set $\{1,2,\ldots, n\}$, so that we have a preferred inclusion $S_{n-1}<S_n$.

A {\bf closed inductive system} is the assignment of a pseudo-abelian subcategory $\mA^n\subset \Rep_{\bk} S_n$, for each $n\in\mZ_{>0}$, so that $\Res^{S_n}_{S_{n-1}}$ sends $\mA^n$ into $\mA^{n-1}$ such that additionally
\begin{itemize}
\item[(i)] Each $\bA^n\subset\Rep S_n$ is a rigid monoidal subcategory.
\item[(ii)]  For all $m,n\in\mZ_{>0}$, the induction product
$$\mA^m \times\mA^n\xrightarrow{\Ind^{S_{m+n}}_{S_m\times S_n}(-\boxtimes -)}\Rep S_{m+n}$$
takes values in $\mA^{m+n}$.
\end{itemize}
One can define inductive systems also as classes of indecomposable representations, and we write $\Indec\bA^n$ for the set of isomorphism classes of indecomposable representations in $\bA^n$.

The minimal example of a closed inductive system is $\Young$, which consists of direct sums of Young modules.
Since there is only one closed inductive system in characteristic zero ($\bA^n=\Rep S_n=\Young^n$), we now assume that $\mathrm{char}(k)=p>0$.

The {\bf signed Young modules} are the indecomposable direct summands of the $\bk S_n$-modules that are induced from one-dimensional representations of Young subgroups (which are exterior tensor products of trivial and sign representations). We have the corresponding closed inductive system $\SYoung$, with $\SYoung=\Young$ if $p=2$, which is the minimal closed inductive system containing the sign representations. In \cite{coulembier_inductive_2024}, the closed inductive system $\UnEn$ was introduced. We call an $S_n$-representation
{\bf unentangled} if its restriction to every Young subgroup $S_\lambda<S_n$ (for $\lambda\vDash n$) is a direct summand of a module $M_1\boxtimes M_2\boxtimes\cdots \boxtimes M_l$ for $S_{\lambda_i}$-representations $M_i$, or equivalently, a direct sum of such modules. The unentangled representations form a closed inductive system $\UnEn$, and we say that an inductive system $\bA$ is unentangled if $\bA\subset\UnEn$. It then follows that, for $i+j=n$, the functor $\Res^{S_n}_{S_i\times S_j}$ sends $\bA^n$ into $\bA^i\otimes_k\bA^j$.

\begin{remark}
    Contrary to the general case of inductive system, also defined in \cite{coulembier_inductive_2024}, there is a unique minimal {\em closed} inductive system containing a specified class of (indecomposable) representations. Focusing on minimal sets of generators then leads to the concept in the following proposition.
\end{remark}

Recall, see~\cite[\S 9]{Alperin}, that a {\bf vertex} of a representation $V$ of a finite group $G$ is a minimal subgroup $H<G$ such that $V$ is a direct summand of a representation induced from $H$ to $G$. A vertex is unique up to conjugation, so that we can often speak of `the' vertex, and is always a $p$-group .

\begin{proposition}\label{prop:seeds}
    Let $\bA$ be an unentangled closed inductive system and consider an indecomposable $M\in\bA^n$, for some $n\in\mZ_{>0}$. Then the following conditions are equivalent:
    \begin{enumerate}
        \item $M$ is not a direct summand of $\Ind^{S_n}_{S_i\times S_j}(V\boxtimes W)$, for $i,j\in\mZ_{>0}$ with $i+j=n$ and $V\in\bA^i$ and $W\in \bA^j$.
        \item $M$ is not a direct summand of any $S_n$-representation induced from a proper Young subgroup.
        \item The vertex of $M$ is transitive.
    \end{enumerate}
    A necessary condition for properties (1)-(3) to be satisfied is that $n=p^m$ for some $m\in\mZ_{>0}$. 
\end{proposition}
\begin{proof}
The implications (3)$\Rightarrow$(2)$\Rightarrow$(1) follow by definition. To prove that (1) implies (3), consider an $S_n$-representation $M$ with non-transitive vertex. It follows from \cite[Proposition~9.1]{Alperin} that $M$ is a direct summand of $\Ind^{S_n}_{S_i\times S_j}\Res^{S_n}_{S_i\times S_j}M$, for some $i+j=n$. The conclusion thus follows from the assumption that $\bA$ be unentangled.

    Condition (3) implies that there exists a transitive $p$-subgroup of $S_n$, which requires $n=p^m$.
\end{proof}

We call the indecomposable representations $M$ in $\bA$ for which the conditions in Proposition~\ref{prop:seeds} are satisfied the {\bf seeds} of $\bA$, since their exterior products canonically generate $\bA$. The obvious example of a seed is the trivial representation $\unit$ of $S_{p^m}$, for every $m$. These are called the {\bf trivial seeds} and they are contained in any inductive system. We denote by $\Sds\bA^n\subset\Indec\bA^n$ the set of (isomorphism classes of) seeds of degree $n$. Hence $\Sds\bA^n=\varnothing$ unless $n$ is a power of $p$.

\begin{example}\label{ex:seeds}
    Assume $p>2$. For the closed inductive system $\SYoung$, the non-trivial seeds are precisely the sign representations of $\Rep S_{p^m}$ for $m\in\mZ_{>0}$, as follows quickly from Proposition~\ref{prop:seeds}.
\end{example}

\subsection{Polynomial functors} Let $k$ be an algebraically closed field.

For a tensor category $\cC$ over $k$ and $X\in\cC$, we have the inductive system $\bB_X[\cC]$ and the closed inductive system $\bB[\cC]=\sum_{Y\in\cC}\bB_Y[\cC]$, as defined in \cite{coulembier_inductive_2024}. By \cite[Theorem~3.1.3(5)]{coulembier_inductive_2024}, $\bB[\cC]$ is unentangled. We have $\bB[\cC]=\bB[\cD]$ if there exists a tensor functor $\cC\to\cD$, and we are thus primarily interested in $\bB[\cC]$ for incompressible $\cC$, by \cite[Theorem~A]{Incomp}.

For the exposition in this section, we make the following hypothesis, which allows us to bypass technicalities, and happens to be satisfied for all tensor categories considered in the current paper: 

\begin{hypo}\label{hypo}We assume that $\Indec\bB^d[\cC]$ is finite, for each $d\in\mZ_{>0}$.
\end{hypo}

In \cite[\S 8.1]{coulembier_inductive_2024}, the category $\SPol^d\cC$ of strict polynomial functors of degree $d$ on $\cC$ was introduced. By Theorem~9.1.1 in \cite{coulembier_inductive_2024} and some simplifications due to Hypothesis~\ref{hypo}, we have equivalences
\begin{equation}\label{eqSPol}\SPol^d\cC\;\simeq\;\cC\boxtimes\Fun_k(\bB^d[\cC],\Vecc)\;\simeq\;\Fun_k(\bB^d[\cC],\cC).\end{equation}
The same theorem also demonstrates that, for any $X\in\cC$, a canonical evaluation functor
$$\SPol^d\cC\;\to\; \Rep^d_{\cC}\GL(X)$$
is an equivalence if and only if $\bB^d_X[\cC]=\bB^d[\cC]$.
The objects $X\in\cC$ that satisfy these equivalent conditions are called {\bf relatively $d$-discerning} with respect to $\cC$. Hypothesis~\ref{hypo} is equivalent to the assumption that a relatively $d$-discerning object exists, for every $d$.

 In \cite[\S 8.2]{coulembier_inductive_2024} a category $\Pol^d$ of universal polynomial functors of degree $d$ is defined, and its objects are coherent families of functors defined on all tensor categories over $k$. The primary example is the functor $X\mapsto X^{\otimes d}$. It is proved that this category is equivalent to a category of $k$-linear functors $\TC^d_{k}\to \Vecc$, where $\TC^d=\sum_{\cD}\bB^d[\cD]$, for $\cD$ ranging over all tensor categories over $k$. Evaluation of universal polynomial functors on $\cC$ factors through the Serre quotient $\Fun_k(\bB^d[\cC],\Vecc)$ of (the category of functors $\TC^d\to\Vecc$ equivalent to) $\Pol^d$. For $k$-linear functors $D,E,F,\ldots:\bB^d[\cC]\to\Vecc$, the corresponding endofunctor of $\cC$ will be denoted by $\mathfrak{D},\mathfrak{E},\mathfrak{F},\ldots$, and symbolically we have
\begin{equation}\label{eqmF}\mathfrak{F}(X)\;=\; F(X^{\otimes d}),\end{equation}
for $X^{\otimes d}$ viewed as an $S_d$-representation in $\cC$. We have $\mF(\lambda f)=\lambda^d\mF(f)$, for $\lambda\in k$ and $f$ a morphism in $\cC$. It is also proved in Theorem 9.1.1 of \cite{coulembier_inductive_2024} that the evaluation
\begin{equation}\label{EqPolyGL}\Fun_k(\bB^d[\cC],\Vecc)\;\to\; \Rep^d_{\cC}(\GL(X),\varepsilon),\quad F\mapsto \mF(X)\end{equation}
is an equivalence if and only if $X$ is relatively $d$-discerning. 

{\em Convention:} Henceforth, with slight abuse of terminology, by `polynomial functor' (on $\cC$), we will refer to a $k$-linear functor $F:\bB^d[\cC]\to\Vecc$. This is the simplest interpretation, but still allows us to think of this more intuitively as a strict polynomial functor via equivalence \eqref{eqSPol}, or as a recipe for a polynomial $\GL(X)$-representations for any $X\in\cC$ via \eqref{EqPolyGL}, or as the functor $\mF:\cC\to\cC$ via \eqref{eqmF}. We will for instance say that `$F$ is additive' or `exact', when $\mF$ is additive or exact in the ordinary sense.

For $M\in\Indec\bB^d[\cC]$ denote by $D_M$ the simple object of $\Fun_k(\bB^d[\cC],\Vecc)$ that has projective cover $\Hom_{S_d}(M,-)$. We summarise some of the discussion above in the following lemma.

\begin{lemma}\label{lem:poly_glx}
The simple polynomial functors of degree $d$ on $\cC$ are precisely $D_M$, with $M\in\Indec\bB^d[\cC]$. For any $X\in\cC$ a complete irredundant list of simple polynomial representations of $GL(X)$ of degree $d$ is given by $\mD_M(X)$, for $M\in\Indec\bB^d_X\subset\Indec\bB^d[\cC]$, while $\mD_M(X)=0$ for $M\in \Indec\bB^d[\cC]$ not included in in $\bB^d_X[\cC]$.
\end{lemma}

\begin{example}\label{ex:Fr+}
    We have
    $$\mD_{\unit}(X)\;=\;\im(\Gamma^dX\hookrightarrow X^{\otimes d}\tto S^dX).$$ 
    If $d=p^m$, then this is the functor $\Fr^{(m)}_+$ from \cite[\S 4.1]{Tann}, which is isomorphic to $(\Fr_+)^m$ when $\cC$ is Frobenius exact, see \cite[Lemma~5.1]{coulembier_frobenius_2022}, with $\Fr_+:=\Fr_+^{(1)}$. 

    Similarly, if $p>2$, the sign representation, leads to $\Fr_-^{(m)}$ in degree $p^m$.
\end{example}

\begin{remark}\label{rem:DTP}
\begin{enumerate}
    \item We can extend \eqref{EqPolyGL} to an equivalence 
    $$\Fun_k(\bB^d[\cC],\Vecc)\boxtimes \cC\;\xrightarrow{\sim}\;\Rep^d_{\cC}\GL(X),$$
    leading to further equivalences using  \eqref{eqSPol}. Hence we could have taken the convention to define $\Fun_k(\bB^d[\cC],\cC)$ (or yet equivalently $\SPol^d\cC$) as our category of polynomial functors, rather than $\Fun_k(\bB^d[\cC],\Vecc)$. Objects in $\Fun_k(\bB^d[\cC],\Vecc)\boxtimes \cC$ can also be evaluated to functors $\cC\to\cC$, for example $\mD_M\boxtimes X_0$ leads to $X\mapsto \mD_M(X)\otimes X_0$.
    \item 
    The simple objects in $\SPol\cC$ are labeled by the Cartesian product of the set of simple objects in $\cC$ and the set of simple polynomial functors as in Lemma~\ref{lem:poly_glx}.
    \item The category $\Rep^d_{\cC}(\GL(X),\varepsilon)$ is not always a Serre subcategory of $\Rep^d_{\cC}\GL(X)$. Indeed, this is clearly the case if there is a self-extension of $\unit$ in $\cC$.
    We can interpret this also via polynomial functors. Inside our category of polynomial functors $\Fun_k(\bB^d[\cC],\Vecc)$, there is never a self-extension of $\mD_{\unit}$. This follows from the fact that the endomorphism algebra of the projective cover is, by the Yoneda Lemma, given by $\bk=\End_{S_d}(\unit)$. On the other hand, if $E$ is a self-extension of $\unit$ in $\cC$, there is a self-extension $\mD_{\unit}(X)\otimes E$ of $\mD_\unit(X)$ in $\Rep^d_{\cC}\GL(X)$.
\end{enumerate}

\end{remark}

It is proved in Proposition~9.5.3 of \cite{coulembier_inductive_2024} that the braiding homomorphism
\begin{equation}\label{eq:IT}\bk S_d\;\to\; \End_{\GL(X)}(X^{\otimes d})\end{equation}
is an isomorphism if and only if it is injective if and only if $\bB^d\subset\Rep S_d$ contains all projective modules. Indecomposable projective modules can be labelled as the Young modules $Y^\lambda$ for $p$-restricted $\lambda\vdash d$, so the condition is equivalent to $\mD_{Y^\lambda}(X)\not=0$ for such $\lambda$. We call $X$ for which \eqref{eq:IT} is an isomorphism {\bf $d$-faithful}, and any relatively $d$-discerning object is $d$-faithful.

We conclude this section with a short explanation on tensor products and compositions of polynomial functors. For polynomial functors $F:\bB^d[\cC]\to\Vecc$ and $G:\bB^m[\cC]\to\Vecc$, we denote by $F\otimes G:\bB^{d+m}[\cC]\to \Vecc$ the following poynomial functor. It takes $M\in \bB^{d+m}[\cC]$ and restricts it to the Young subgroup $S_d\times S_m<S_{d+m}$, which yields an object in (the additive completion of) $\bB^m[\cC]\otimes \bB^d[\cC]$, by the fact $\bB[\cC]\subset\UnEn$, so that we can apply $F$ and $G$ on both factors. Perhaps more importantly, we can define this polynomial functor via the equivalence \eqref{EqPolyGL}. Indeed, $\Rep_{\cC}(\GL(X),\varepsilon)$ is a tensor category and the polynomial representation form a monoidal subcategory. If $\mF\otimes \mG$ is the endofunctor of $\cC$ corresponding to $F\otimes G$, then $(\mF\otimes \mG)(X)=\mF(X)\otimes \mG(X)$. There is also a polynomial functor
$F\circ G: \bB^{dm}[\cC]\to\Vecc$. Under equivalence \eqref{EqPolyGL}, it corresponds to the $\GL(X)$-representation $\mF(\mG(X))$, {\it i.e.} the composition
$$\GL(X)\to \GL(\mG(X))\to \GL(\mF(\mG(X))),$$ so the
endofunctor of~$\cC$ corresponding to $F\circ G$ is given by the ordinary composition $\mF\circ\mG$.

\subsection{Additive polynomial functors} \label{sec:additive}

\begin{proposition}\label{prop:add}
    The polynomial functors of degree $d$ on $\cC$ that are additive form the Serre subcategory of $\Fun_k(\bB^d[\cC],\Vecc)$ of functors that vanish on all of $\Indec\bB^d[\cC]\backslash\Sds\bB^d[\cC]$.
\end{proposition}
\begin{proof}
    Inside $\cC\boxtimes\Rep_{\bk}S_d$, which can be interpreted as the category of $S_d$-representations in $\cC$, we have, for $X,Y\in\cC$
    $$(X\oplus Y)^{\otimes d}\;\simeq\; X^{\otimes d}\oplus Y^{\otimes d}\oplus \bigoplus_{0<i<d}\Ind^{S_d}_{S_i\times S_{d-i}}(X^{\otimes i}\otimes Y^{\otimes d-i}).$$
    For a $k$-linear functor $F:\bB^d[\cC]\to\Vecc$, the resulting endofunctor $\mF$ of $\cC$ will thus be additive if and only if it vanishes on all (direct summands of) representations $\Ind^{S_d}_{S_i\times S_{d-i}}(V\boxtimes W)$ for $V\in \bB^i[\cC]$ and $W\in \bB^{d-i}[\cC]$. The conclusion follows from characterisation (1) in Proposition~\ref{prop:seeds}.
\end{proof}

\begin{remark}\label{rem:poly_vanishing}
\begin{enumerate}
    \item Due to the behaviour $\mF(\lambda f)=\lambda^d\mF(f)$, non-zero additive polynomial functors can only exist in degrees $d$ that are powers of $p$. This also follows immediately from the combination of Propositions~\ref{prop:seeds} and~\ref{prop:add}.
    \item For a general (not necessarily additive) polynomial functor of degree $d$, we have
    $$\mF(X\oplus Y)\;\cong\; \mF(X)\oplus \mF(Y)\oplus \bigoplus_{0<i<d}\mF_i(X,Y),$$
    where the $\mF_i$ can be interpreted as `polynomial bifunctors' of degree $i,d-i$. For us it will be more important to interpret $\mF(X\oplus Y)$ as a representation over $\GL(X)$, via the restriction $\GL(X)<\GL(X\oplus Y)$. Then $\mF(X)$ is the (maximal) summand of degree $d$, $\mF_i(X,Y)$ the summand of degree $i$ and $\mF(Y)$ the summand of degree $0$. Degree refers to the eigenvalue of the action of $\mathbb{G}_m<\GL(X)$.
    \item The classification of additive polynomial functors in $\Fun(\bB^d[\cC],\cC)$ follows from that in $\Fun(\bB^d[\cC],\Vecc)$. Indeed, a functor in the former functor category is `additive' if and only if its composition with $\omega:\cC\to\Vecc$ is `additive' in the latter category. Here $\omega$ is some faithful exact $k$-linear functor (not necessarily monoidal). Equivalently, the category of additive functors in $\Fun(\bB^d[\cC],\cC)$ is the Deligne tensor product of $\cC$ with the category of additive functors in $\Fun(\bB^d[\cC],\Vecc)$.
    \item Similarly, a functor $F\in\Fun(\bB^d[\cC],\cC)$ is `exact' if and only if $\omega\circ F\in\Fun(\bB^d[\cC],\Vecc)$ is `exact'. Indeed, the evaluation of $F$ to a functor $\cC\to\cC$ is given by the composite
    $$\cC\xrightarrow{-^{\otimes d}}\cC\boxtimes\Rep S_d\xrightarrow{\cC\boxtimes F}\cC\boxtimes\cC\xrightarrow{\otimes}\cC,$$
    where the last functor is exact and faithful, and thus does not influence overall exactness.
    \item By Proposition~\ref{prop:add}, the classification of additive polynomial functors can be read off from $\bB[\cC]$. The same does not seem to be true for exact functors.
\end{enumerate}
    
\end{remark}

The coordinate algebra $k[\GL(X)]$ of $\GL(X)$ is a quotient algebra of $\Sym(X^\ast\otimes X)\otimes \Sym(X\otimes X^\ast)$, such that 
$$M_X:=\Sym(X^\ast\otimes X)=\Sym(\uEnd(X))\subset k[\GL(X)]$$ is a subalgebra and subcoalgebra (but not a Hopf subalgebra). The comodule category of the bialgebra $M_X$ is precisely the category of polynomial representations of $\GL(X)$. Every subobject $M^d_X:=\Sym^d(\uEnd(X))$ is a subcoalgebra, and its comodule category is precisely the category of polynomial $\GL(X)$-representations of degree $d$.

We denote the coproduct on $M_X$, that is inherited from $k[\GL(X)]$ as discussed in the previous paragraph, by $\Delta_{\times}$. It restricts to a coproduct $M_X^d\to M_X^d\otimes M_X^d$. We can also interpret $M_X$, just like any $\Sym(Y)$ for $Y\in\cC$, as a Hopf algebra, corresponding to the additive group scheme of $Y$. We denote the corresponding coproduct on $M_X$ by $\Delta_+$. This one restricts to $M_X^d\to \oplus_{i+j=d}M_X^i\otimes M_X^j$. There is a `distributivity' compatibility between both coproducts. Concretely, $M_X$ represents the `affine algebraic ring' internal to $\cC$ given by the functor from the category of commutative algebras in $\Ind\cC$ to the category of rings
$$\mathsf{CAlg}\cC\;\to\;\mathsf{Ring},\quad A\mapsto \End_A(A\otimes X).$$
Diagrammatically this translates into the following commutative diagram, as well as its `right-hand' version

\[\begin{tikzcd}
	{M_X} && {M_X\otimes M_X} & {M_X\otimes M_X\otimes M_X} \\
	{M_X\otimes M_X} && {M_X\otimes M_X\otimes M_X\otimes M_X}
	\arrow["{\Delta_{\times}}", from=1-1, to=1-3]
	\arrow["{\Delta_+}"', from=1-1, to=2-1]
	\arrow["{\Delta_+\otimes M_X}", from=1-3, to=1-4]
	\arrow["{\Delta_\times\otimes \Delta_\times}", from=2-1, to=2-3]
	\arrow[from=2-3, to=1-4]
\end{tikzcd}\]
where the right upwards arrow is multiplication of the second and fourth factor, ending up in the third factor.

\begin{lemma}
    The subobject $C_X^d\subset M^d_X$ defined as the kernel of the composite
    $$M^d_X\;\xrightarrow{\Delta_+}\;\bigoplus_{0\le i\le d}M^i_X\otimes M^{d-i}_X\;\to\;\bigoplus_{0< i< d}M^i_X\otimes M^{d-i}_X,$$ where the second map is projection onto a direct summand, is a sub-coalgebra of $(M_X^d,\Delta_{\times})$.
\end{lemma}
\begin{proof}
    This is a direct consequence of the distributivity between $\Delta_+$ and $\Delta_{\times}$.
\end{proof}

\begin{example}
    If $\cC=\Vecc$, then $M_{\unit^n}=k[x_{ij}]$ is a polynomial algebra in $n^2$ variables and the comultiplications are defined by $\Delta_+(x_{ij})=x_{ij}\otimes 1+1\otimes x_{ij}$ and $\Delta_{\times}(x_{ij})=\sum_l x_{il}\otimes x_{lj}$. Hence, for $d>1$
    $$C^d_{\unit^n}\;=\;\begin{cases}
    \Span\{x_{ij}^d\},&\mbox{if $d$ is a power of $\mathrm{char}(k)$},\\
    0,&\mbox{otherwise.}
    \end{cases}$$
\end{example}

The following characterisation of additive polynomial functors might well remain valid for general tensor categories. For now, we focus on a special setting that is powerful enough to include the first cases of interest, namely $\Ver_4^+$ and $\Ver_p$. In general, the methods below only show that additive polynomials of degree $d$ are in bijection with a {\em subset} of $C^d_X$-comodules.

\begin{theorem} \label{thm:cd_comodules}Assume that $\cC$ admits a, not necessarily symmetric, tensor functor to a semisimple, not necessarily symmetric, tensor category.
Let $X\in\cC$ be relatively $d$-discerning.
    Under equivalence~\eqref{EqPolyGL}, the category of additive polynomial functors is identified with the category of comodules in $\cC$ of $C^d_X\subset k[\GL(X)]$.
\end{theorem}
\begin{proof}
    By Proposition~\ref{prop:add}, the category of additive polynomial functors is the Serre subcategory of functors that are sent to zero under the restriction
    \begin{eqnarray*}\Fun_k(\bB^d[\cC],\Vecc)&\to&\bigoplus_{0<i<d}\Fun_k(\bB^i[\cC]\otimes_{k}\bB^{d-i}[\cC],\Vecc)\\
    &\simeq& \bigoplus_{0<i<d}\Fun_k(\bB^i[\cC],\Vecc)\boxtimes \Fun_k(\bB^{d-i}[\cC],\Vecc) \end{eqnarray*}
    along
    $$\bB^i[\cC]\otimes_{k}\bB^{d-i}[\cC]\;\xrightarrow{\Ind^{S_d}_{S_{i}\times S_{d-i}}(-\boxtimes -)}\;\bB^d[\cC].$$

For each $i$, equivalence \eqref{EqPolyGL} then induces an exact functor
$$\Rep^d(\GL(X),\varepsilon)\;\to\; \Rep^i(\GL(X),\varepsilon)\boxtimes \Rep^{d-i}(\GL(X),\varepsilon),$$
and our task is to identify the (co)modules that are sent to zero with the ones stated in the theorem. To describe the functor, we omit the $\varepsilon$-restriction.  In terms of the subcoalgebra $M^d_X\subset k[\GL(X)]$, the functor then corresponds to a composite
$$\Comod(M^d_X)\xrightarrow{\sim} \Comod(M^d_{X_1\oplus X_2}) \to \Comod(M^{i}_{X_1}\otimes M^{d-i}_{X_2}),$$
where $X_1,X_2$ are just copies of $X$, labelled differently for clarity. To obtain the precise equivalence we can work with the equivalence \eqref{EqPolyGL}. Indeed, $M^d_X=\mF(X)$ for $\mF=\Sym^d(\uHom(X,-))$, so $M^d_X$ is sent to $B:=\Sym^d(\uHom(X,X_1\oplus X_2))$ in $\Comod(M^d_{X_1\oplus X_2})$, and it follows that
the first functor is the Morita equivalence realised by cotensoring over the bi-comodule $B$ and the second a restriction (corresponding to $GL(X_1)\times \GL(X_2)<\GL(X_1\oplus X_2)$) followed by the projection onto suitable degrees. Overall, the functor is given by
$$\left(\Sym^i(\uHom(X,X_1))\otimes \Sym^{d-i}(\uHom(X,X_2))\right)\;\Box^{M^d_X}\;-\;\,\simeq\;\, \left(M^i_X\otimes M^{d-i}_X\right)\;\Box^{M^d_X}\;-,$$
where the second interpretation is simply forgetting the relabellings $X_1,X_2$ of $X$.
The right $M^d_X$-comodule structure on $M^i_X\otimes M^{d-i}_X$ for the above cotensor is such that the map $M_X^d\to M_X^i\otimes M_X^{d-i}$ (coming from $\Delta_+$) is a morphism of $M_X^d$-comodules, which is again a consequence of distributivity. 

It now suffices to show that for an $M^d_X$-comodule $V$ in $\cC$, the composite
$$V\otimes V^\ast\;\to\; M^d_X\;\to\;M^i_X\otimes M^{d-i}_X$$
is zero if and only if
$$(M^i_X\otimes M^{d-i}_X)\;\Box^{M^d_X}\;V\;=\;0.$$
We can prove instead the dual formulation in terms of algebras. Hence by Lemma~\ref{lem:AN} below, we need to prove that
$$\uHom_{S_d}(X^{\otimes d},\Ind^{S_d}_{S_i\times S_{d-i}}X^{\otimes d})\otimes \uHom_{S_d}(X^{\otimes d},\Ind^{S_d}_{S_i\times S_{d-i}}X^{\otimes d})\to \uHom_{S_d}(X^{\otimes d},\Ind^{S_d}_{S_i\times S_{d-i}}X^{\otimes d})$$
is an epimorphism, where the map is induced from the canonical epimorphism (split in $\cC$)
$$\Ind^{S_d}_{S_i\times S_{d-i}}X^{\otimes d}\;\tto\; X^{\otimes d}$$
and composition.
Now, if $\cC$ has a (not necessarily symmetric) tensor functor to a semisimple tensor category $\cC_1$, it suffices to prove that image of the morphism is an epimorphism in $\cC_1$. In $\cC_1$, the $S_d$-representation $X^{\otimes d}$ can be written as
$\oplus_i V_i\boxtimes L_i,$
where $L_i$ runs over a finite set of simple objects in $\cC_1$ and $V_i$ are $S_d$-representations. It then follows that we need to prove that the morphism in vector spaces of the form
$$\bigoplus_j\left(\Hom_{S_d}(V_j, \Ind^{S_d}_{S_i\times S_{d-i}}V_l)\otimes \Hom_{S_d}(V_i, \Ind^{S_d}_{S_i\times S_{d-i}}V_j)\right)\;\to\;  \Hom_{S_d}(V_i, \Ind^{S_d}_{S_i\times S_{d-i}}V_l),$$
given by $f\otimes g\mapsto f\circ \pi\circ g$ for the canonical projection $\pi:\Ind^{S_d}_{S_i\times S_{d-i}}V_j\tto V_j $,
is an epimorphism for every $i,l$. That it is an epimorphism follows from the fact that $X$ is relatively $d$-discerning and Lemma~\ref{lem:GH} below by picking $W_1=V_i$, $W_2=\oplus_j V_j$ and $W_3=V_l$. 
    \end{proof}

    \begin{lemma}\label{lem:AN}
        Let $A$ be an algebra in a tensor category $\cC$, with $N\in\cC$ a right $A$-module, with action $\rho:N\otimes A\to N$, equipped with an $A$-module morphism $a:N\to A$. Assume that the composite
        $$N\otimes N\xrightarrow{N\otimes a} N\otimes A\xrightarrow{\rho} N$$
        is an epimorphism. Then for any left $A$-module $V\in\cC$, the object
        $N\otimes_A V$ is zero if and only if the composite $N\to A\to\uEnd(V)$ is zero.
    \end{lemma}
    \begin{proof}
        One direction does not require the epimorphism. For the `if' claim we compose the defining diagram for $N\otimes_AV$ with $N\otimes a\otimes V$, yielding an equalising diagram
        $$N\otimes N\otimes V\rightrightarrows N\otimes V\tto N\otimes_AV.$$
        Now, by assumption one of the two parallel arrows is an epimorphism, while the other is zero when $N\to A\to\uEnd(V)$ composes to zero, leading to $N\otimes_AV=0$.
    \end{proof}

    \begin{lemma}\label{lem:GH}
        Let $G$ be a finite group, with subgroup $H<G$ and $W_1,W_2,W_3\in \Rep_kG$. Assume that every indecomposable direct summand of $\Ind^G_HW_3$ (we omit $\Res^G_H$ from notation) also appears as a summand in $W_2$. Consider the canonical projection $\pi:\Ind^G_HW_2\tto W_2$. Then the morphism
        $$\Hom_G(W_2,\Ind^G_HW_3)\otimes \Hom_G(W_1,\Ind^G_HW_2)\xrightarrow{f\otimes g\mapsto f\circ \pi\circ g}\Hom_G(W_1,\Ind^G_HW_3)$$
        is surjective.
    \end{lemma}
    \begin{proof}
        By assumption, there exists $n$ for which there are morphisms $z:\Ind^G_HW_3\to W^n_2$ and $y: W_2^n\to\Ind^G_HW_3$ so that $y\circ z$ is the identity. Now, by relative $H$-projectivity of $\Ind^G_H W_3$ (\cite[Proposition 9.1]{Alperin}) and the fact that $\pi$ splits over $H$, there exists a lift $x:\Ind^G_HW_3\to (\Ind^G_HW_2)^n$ of $z$, meaning $z=\pi^n\circ x$. Hence $y\circ \pi^n\circ x$ is the identity morphism of $\Ind^G_HW_3$. In other words, we can write any $h:W_1\to\Ind^G_H W_3$ in the target as
        $$\sum_i y_i\circ \pi\circ (x_i\circ h),$$
        for $x_1,\ldots, x_n$, resp. $y_1,\ldots, y_n$ the coordinates of $x$ resp. $y$.
    \end{proof}

\subsection{Example: super vector spaces}
Let $k$ be an algebraically closed field of characteristic $p>2$. We apply the concepts and results from the current section to the case $\cC=\Vecc$.

We have $\bB[\sVec]=\SYoung$, as follows from definitions, see \cite[Theorem~4.1.1]{coulembier_inductive_2024}.
Additive superpolynomial functors were classified in \cite[Theorem~3.5]{Giordano}. The classification now also follows from our general framework.
Note that parity shift functors do not appear in our classification, due to our conventions. In order to obtain the precise classification of \cite{Giordano}, one has to apply Remarks~\ref{rem:DTP}(2) and \ref{rem:poly_vanishing}(3), where the odd line in $\cC=\sVec$ corresponds to the parity shift functor.
\begin{corollary}
    The additive polynomial functors on $\sVec$ in degree $p^j$ are given by the direct sums of $\Fr_+^j$ and $\Fr_-^j$.
\end{corollary}
\begin{proof}
    This follows immediately from Examples~\ref{ex:seeds} and~\ref{ex:Fr+} and Proposition~\ref{prop:add}
\end{proof}

It was determined in \cite{CEKO} when a supervector space $\unit^{m|n}=\unit^m\oplus \bar{\unit}^n$ is $d$-faithful.
The answer to this problem, as well as the question of when the space is relatively $d$-discerning, is invariant under $m\leftrightarrow n$. We can thus for convenience assume that $m\ge n$. We also exclude the classical even case, so we assume $m\ge n>0$.
We set first
$$c=c(m,n):=\frac{p+n-m}{2}\;\in\;\mZ[1/2].$$
The main result from \cite{CEKO} then states that $\unit^{m|n}$ is $d$-faithful if and only if we are in one of the following three situations:
$$\begin{cases}
c<2,&\mbox{and}\quad d\le m+(p-1)n, \quad\qquad\mbox{or}\\
2\le c\le n&\mbox{and}\quad d\le \lceil (n+1)p-1-c^2\rceil, \qquad\mbox{or}\\
c>n&\mbox{and}\quad d\le m+n+mn.
\end{cases}$$

The sufficient condition ``$m\ge d\le n$'' for $\unit^{m|n}$ to be relatively $d$-discerning was observed in \cite[Theorem~4.2]{Axtell}. Now we can derive the sufficient and necessary condition:
\begin{proposition}
    A supervector space $\unit^{m|n}$ with $m\ge n>0$ is relatively $d$-discerning in $\sVec$ if and only if we are in one of the following four situations:
$$\begin{cases}
c\le 0,&\mbox{and}\quad d\le (n+1)p-1, \qquad\quad\mbox{or}\\
0< c<2,&\mbox{and}\quad d\le m+(p-1)n,\quad \qquad\mbox{or}\\
2\le c\le n,&\mbox{and}\quad d\le \lceil (n+1)p-1-c^2\rceil, \qquad\mbox{or}\\
c>n,&\mbox{and}\quad d\le m+n+mn.
\end{cases}$$
\end{proposition}
\begin{proof}
    The indecomposable signed Young modules were classified in \cite{Donkin} and are labelled as $Y^{\lambda\mid p\mu}$ for partitions $\lambda,\mu$ such that $|\lambda|+p|\mu|=d$. Here $Y^{\lambda}=Y^{\lambda\mid 0}$ are the Young modules.

    A necessary condition for $\unit^{m|n}$ to be relatively $d$-discerning is that it is $d$-faithful, \textit{i.e.} that $\mD_{Y^\lambda}(\unit^{m|n})$ is non-zero for every $p$-restricted $\lambda\vdash d$. However, in \cite[Proposition~2.2.8(i)]{CEKO}, it was proved that when $\unit^{m|n}$ is $d$-faithful, then it actually follows immediately that $\mD_{Y^\lambda}(\unit^{m|n})$ is non-zero for every $\lambda\vdash d$. It remains to be verified for which $d$-faithful $\unit^{m|n}$ we have $\mD_{Y^{\lambda\mid p\mu}}(\unit^{m|n})$ for non-zero $\mu$. It follows from the procedure in the proof of \cite[Theorem~6.5]{Brundan-Kujawa}, or from the Steinberg tensor product theorem as discussed in \cite[Remark~4.6]{Brundan-Kujawa}, that $\mD_{Y^{\lambda\mid p\mu}}(\unit^{m|n})$ is non-zero precisely when both $\mD_{Y^{\lambda}}(\unit^{m|n})$ and 
    $$\mD_{Y^{0\mid p\mu}}(\unit^{m|n})\;\simeq\; \mD_{Y^{0\mid p\mu}}(\unit^{0|n})$$ are non-zero.
    The latter vanishing is entirely classical and the displayed objects are zero if and only if the length of $\mu$ exceeds $n$. This condition is satisfied for all relevant $\mu$ if and only if $d<p(n+1)$. The conclusion then follows from observing that the latter condition is automatically satisfied for $d$-faithful objects when $c>0$, but not when $c\le 0$.
    \end{proof}

\begin{remark}
 It follows that there are supervector spaces that are $d$-faithful without being $d$-discerning. To give a concrete example, when 
    $$m=n+p\qquad\mbox{and}\qquad d=m+(p-1)n=p(n+1),$$ then $\unit^{m|n}$ is $d$-faithful, but not relatively $d$-discerning.

\end{remark}

\section{Polynomial Functors in $\Ver_4^+$}\label{sec-Ver4}

Let $k$ be an algebraically closed field of characteristic 2.

\subsection{A Frobenius twist}
\label{sec:FrobTwist}

In \cite[\S 4]{CS25}, a notion of Frobenius twist of schemes in tensor categories was introduced. We make this explicit for $\Ver_4^+$. Let $\cC$ first be an arbitrary tensor category over a field of characteristic $p$. For $A\in \CAlg(\cC)$, we define the subalgebra $A^{[1]}\subset A$ as the image of the multiplication morphism $\Gamma^pA\to A$. This leads to a quotient group $G\to G^{[1]}$ for an affine group scheme $G$ (which is for $\cC=\Vecc$ is the image of the usual Frobenius twist $G\to G^{(1)}$) with $k[G^{[1]}]:=k[G]^{[1]}$. 

We return to $\cC=\Ver_4^+$. Below we write $V^{(1)}$ for the usual Frobenius twist (for instance defined via an extension of scalars or as the image of $\Gamma^2V\to  S^2 V$) of a vector space $V$.

\begin{lemma}
    For vector spaces $V=k^m$ and $W=k^n$, set $G=\GL(V\oplus W\otimes P)$. Then 
    $$G^{[1]}\;\simeq\; (\GL(V)\times\GL(W)) \ltimes (V^{(1)}\otimes W^\ast\times W\otimes (V^{(1)})^\ast\times W\otimes W^\ast),$$
    an affine group scheme over $k$. Applied to $A\in\CAlg(\Ver_4^+)$, the group homomorphism $G(A)\to G^{[1]}(A)$ becomes
    $$\left(\begin{array}{ccc}
    F & C& C'  \\
    B' & D&D'\\
    B&E&D+E'
\end{array}\right)\mapsto \left(\begin{array}{ccc}
    \tilde{F} & \tilde{\tilde{C}}& 0  \\
    0 & \tilde{\tilde{D}}&0\\
    \tilde{\tilde{B}}&\tilde{\tilde{E}}&\tilde{\tilde{D}}
\end{array}\right)$$
where $\tilde{\cdot}$ takes a matrix and squares all entries.
\end{lemma}

Note that the group law on the right-hand side of the group homomorphism is not matrix multiplication, but dictated by the expression of $G^{[1]}$.

\begin{proof}By a direct computation we have $R^{[1]}=k[x^4]$ for $R= \Sym(P)=k[x,y]/y^2$. By \cite[Lemma~4.4(1)]{CS25}, $R^{[2]}$ for $R$ a tensor product of algebras is just the product of the images of the factors, and by \cite[Lemma~4.4(4)]{CS25}, taking $-^{[1]}$ commutes with localisation. We can thus compute $k[\GL(X)]^{[1]}$ directly from the explicit description in \S \ref{GLX:intro}.
\end{proof}

  If we are interested in simple representations, we can compose $G\to G^{[1]}$ with a projection onto the reductive factor 
$$\Frob:\GL(m+nP)\to \GL(m)\times \GL(n), \quad \left(\begin{array}{ccc}
    F & C& C'  \\
    B' & D&D'\\
    B&E&D+E'
\end{array}\right)\mapsto (\tilde{F},\tilde{\tilde{D}}).$$
For a simple $\GL(m)\times \GL(n)$-representation $L(\lambda)\boxtimes L(\mu)$, we can then associate the simple $\GL(m+nP)$-representation
$$\Frob^\ast(L(\lambda)\boxtimes L(\mu))\;\simeq\; L(2\lambda|\chi^\mu).$$
Indeed, the highest weight can be reduced to the cases $G=\mathbb{G}_m$ and $G=\GL(P)$ and simplicity follows from the inclusion $k[G^{[1]}]\subset k[G]$.

\subsection{Odd reflections in characteristic two} \label{section-reflections}

We use a limited notion of reflecting a highest weight of $\GL(m + nP)$. Let $\mathbf T$ denote the set of simple $\GL(P)$-representations, which is $\mathbb Z$-graded via $\deg: \mathbf T \to \mZ$; these are described in Section \ref{ver4_intro}. Recall that the highest weights of irreducible finite-dimensional $\GL(m + nP)$ representations correspond to pairs $\lambda, \Lambda$ where $\lambda\in\mZ^m$ is a nonincreasing sequence and $\Lambda = (\Lambda_1, \dots, \Lambda_n) \in \mathbf{T}^n$ with $\deg \Lambda_i$ a nonincreasing sequence. We will denote such a pair as $\lambda | \Lambda$ and write the corresponding simple representation as $L(\lambda|\Lambda)$. An arbitrary weight of $\GL(m + nP)$, meanwhile, lies in $\mZ^m \times \mathbf T^n$.

For a partition $\lambda$, we will use the usual notion of length $\ell(\lambda)$. Similarly, for $\Lambda\in\bT^n$ for which $\deg\Lambda$ is a partion, {\it i.e.} $\deg\Lambda_1\ge\cdots \ge \deg\Lambda_n\ge0$, we write $\ell(\bT)\le n$ for the largest number $\ell$ with $\deg\Lambda_{\ell}\not=0$.

\begin{definition}
Let $S = \{x_1, \dots, x_{m + n}\}$ generate $m + nP$, so $$\{x_1, \dots, x_m, x_{m + 1}, x_{m + 1}', \dots, x_{m + n}, x_{m + n}'\}$$ is a basis for $m + nP$. 
Let $\sigma \in S_{m + n}$ be a permutation such that $\sigma$ preserves the ordering of the $\{x_i\mid i\le m\}$ and the $\{x_i\mid i\ge m\}$ respectively, i.e. it is a minimal length representative of an $S_m \times S_n$ coset in $S_{m + n}$.
Denote by $B_\sigma$ the Borel subgroup of upper triangular matrices with respect to the basis $\sigma(S)$.
\end{definition}

In the $1 + P$ case, denote by $B_+(P + 1) = B_{(12)}$, the Borel subgroup of upper triangular matrices with respect to $P + 1$.

\begin{definition}
    For $\lambda|T\in\mZ\times\bT$, denote by $R(\lambda|T)\in\mZ\times\bT$ the highest weight of the $\GL(1+P)$-representation $L(\lambda|T)$ with respect to the Borel $B_+(P + 1)$. That is, $R(\lambda|T)$ is the lowest weight of $L(\lambda|T)$ with respect to the Borel $B_+(1 + P)$.
\end{definition}

Note that because $L(2|\1)$ and $L(0|\chi)$ are characters of $\GL(1 + P)$, if $R(\alpha|T) = \beta|T'$ for $\alpha, \beta \in \mbb{Z}$ and $T, T'$ irreducible $\GL(P)$-representations, then $R(2k + \alpha|\chi^\ell T) = 2k + \beta|\chi^\ell T'$. Hence it suffices to determine $R(\lambda|T)$ for $\lambda = 0, 1$ and $0 \le \deg T < 4$. In \cite{hu_supergroups_2024} 4.2.2, we determined the weight decompositions of $\GL(1 + P)$ representations; so we have
\begin{align}
\label{eq:r_def}
    R(0|\1) = 0|\1, &\quad R(1|\1) = 0|T_1\notag \\
    R(0|\xi) = -1|\xi T_1, &\quad R(1|\xi) = 1|\xi\notag \\
    R(0|T_1) = -2|\xi T_3, &\quad R(1|T_1) = 0|T_2\notag \\
    R(0|\xi T_1) = -1|T_2, &\quad R(1|\xi T_1) = -1|T_3 \\
    R(0|T_2) = -1|\xi T_3, &\quad R(1|T_2) = 0|T_3\notag \\
    R(0|T_3) = -2|\xi T_1, &\quad R(1|T_3) = 0|\chi\xi\notag \\
    R(0|\xi T_3) = -1|\chi, &\quad R(1|\xi T_3) = -1|\chi T_1.\notag
\end{align}

\begin{definition}
Let $R_{ij}:\mbb{Z}^m \times \mathbf T^n \to \mbb{Z}^m \times \mathbf T^n$, for $1\le i\le m$ and $1\le j\le n$, be the function that sends $\lambda|\Lambda \mapsto \lambda' | \Lambda'$ where $\lambda'_i | \Lambda_j' = R(\lambda_i|\Lambda_j)$ and all other components are the same.
\end{definition} 

\begin{proposition}
    Let $L$ be an irreducible $\GL(m + nP)$-representation with highest $B_\sigma$-weight $\lambda|\Lambda$ and $1 \le i \le m$, $1 \le j \le n$ be such that $\sigma(i) - \sigma(m + j) = - 1$. Then $R_{ij}(\lambda|\Lambda)$ is the highest $B_{(\sigma(i), \sigma(m + j))\sigma}$-weight of $L$.
\end{proposition}
\begin{proof}
    Consider the parabolic subgroup $B'$ of upper triangular matrices fixing the partial flag $F_0 \subset F_1 \subset \cdots \subset F_{m + n - 1} = m + nP$ with $F_\ell = \langle \sigma(S)_1, \dots, \sigma(S)_\ell\rangle$ for $\ell < \sigma(i)$ and $F_\ell = \langle \sigma(S)_1, \dots, \sigma(S)_{\ell + 1}\rangle$ otherwise. Using $(\sigma(i), \sigma(i) + 1)\sigma$ instead of $\sigma$ then yields the same flag. Let $\lambda'|\Lambda'$ denote the highest weight of $L$ with respect to $B_{(\sigma(i), \sigma(i) + 1)\sigma}$.
    
    Both $B_\sigma, B_{(\sigma(i), \sigma(i) + 1)\sigma} \subset B'$; the former consists of upper triangular matrices in the $1+P$ block (corresponding to $x_i, Y_j$) while the latter consists of lower triangular matrices in the $1 +P$ block. Hence, the $B'$-highest weight of $L$ is also $\lambda | \Lambda$, except that $\lambda_i | \Lambda_j$ is interpreted as the $\GL(1 + P)$-representation with highest $B_+(1+P)$ weight $\lambda_i | \Lambda_j$. Therefore, $\lambda'|\Lambda'$ has the same components as $\lambda | \Lambda$, except in the $i$th and $m + j$th components, where instead of $\lambda_i | \Lambda_j$ we have $R(\lambda_i | \Lambda_j)$. So $\lambda'|\Lambda' = R_{ij}(\lambda|\Lambda)$.
\end{proof}

This implies that one can compute the highest weight of a $\GL(m + nP)$-representation with respect to any Borel that preserves the order of the $x_i$ and the $Y_i$ simply by repeatedly applying a reflection for weights of $\GL(1 + P)$-representations; in particular, to $\GL(nP + m)$, by first applying $R_{m, 1}, \dots, R_{1, 1}$, then $R_{m, 2}, \dots, R_{1, 2}$, and so on.

\subsection{Simple polynomial functors}
We will classify simple polynomial functors on $\Ver_4^+$, by first using the representation theory of $\GL(m + nP)$ described in \cite{hu_supergroups_2024} to enumerate the simple polynomial representations of $\GL(m + nP)$ for $m\gg $ and then applying Lemma~\ref{lem:poly_glx}. In this section, when we say ``polynomial representation" or ``polynomial functor", we mean a {\em simple} polynomial representation or functor respectively.

We first describe simple polynomial representations explicitly for $\GL(P)$ and $\GL(1 + P)$.
\begin{proposition}\label{prop:GLP}
  The simple polynomial representations of $\operatorname{GL} (P)$ are $T_i$ for $i \geqslant
  1$, $\chi^i$ for $i \geqslant 0$, $\xi \chi^i$ for $i \geqslant 1$, $\xi
  T_{4 i + 1}$ for $i \geqslant 1$, and $\xi T_{4 i + 3}$ for $i \geqslant 0$.
  That is, for $n \geqslant 2$, all irreps of degree $n$ are polynomial. When
  $n = 0, 1$, only $\1$ and $T_1$ are polynomial.
\end{proposition}
\begin{proof}
    We interpret the polynomial representations as those that are subquotients of $P^{\otimes n}$. We have $T_1 = P$. The composition series of tensor products of the simple representations of $\GL(P)$ are described in \cite{hu_supergroups_2024}; from these, we can verify that the listed representations are the only ones that are subquotients of $P^{\otimes n}$.
\end{proof}

\begin{proposition}\label{prop:GL1P}
  The simple polynomial representations of $\operatorname{GL} (1 + P)$ are $L (n| T), n
  \geqslant 2$ when $T \neq \xi, \xi T_1$; $L (1| T)$ is polynomial when $T =
  T_n, \xi \chi^n, \chi^n$; and $L (0| T)$ is when $T = \chi^n$.
\end{proposition}

\begin{proof}
  Using the weight space decompositions in \cite{hu_supergroups_2024}, we see that any other
  representations cannot be polynomial -- they either have weights of negative
  degree or ones with $\xi, \xi T_1$, with the latter prohibited by Proposition~\ref{prop:GLP}. By using the tensor product table, we
  determine that the listed representations all lie in subquotients of $L (1,
  T_1)$.
\end{proof}

\begin{theorem}\label{thm:classif rep}
  When $m > d$, the polynomial representations of degree $d$ of $\operatorname{GL} (m
  + nP)$ are L$(\lambda | \chi^{\mu})$ where $\lambda, \mu$ are partitions such that $|\lambda| + 4|\mu| = d$ and the length of $\mu$ is at most $n$.
\end{theorem}
The theorem will be proved in a number of lemmas.

\begin{lemma}\label{lem:w2_irrep}
    Let $X = V \oplus W \otimes P$ with $V = m \1$, $W = n \1$. Then for $j \ge 1$, $W^{(j + 1)}$ is the irreducible representation of $\GL(X)$ with highest weight $0|\chi^{2^{j - 1}}$.
\end{lemma}
\begin{proof}

    Let $y_1, \dots, y_m, x_1, \dots, x_n, x_1', \dots, x_n'$ be a basis of $X$ with $\delta x_i = x'_i$. Then $$\operatorname{span} (y_i^{2^{j + 1}}) = V^{(j + 1)}, \quad \operatorname{span} (y_i^{2^{j + 1}}, x_j^{2^{j + 1}}) =
  V^{(j + 1)} \oplus W^{(j + 1)}$$ are both closed under the $\operatorname{GL} (m + nP)$ action, so $W^{(j + 1)}$ is also a $\GL(X)$-representation. The $\GL(X)$-action is given by
  \begin{equation*}
      \begin{pmatrix} F & C & C' \\ B' & D & D' \\ B & E & D + E' \end{pmatrix} \mapsto D^{(j + 1)}.
  \end{equation*}
  Hence, $W^{(j + 1)}$ is irreducible and has highest weight $0|\chi^{2^{j - 1}}$. Alternatively, we can apply the Frobenius twist at the end of \S \ref{sec:FrobTwist}, to realise $L(0|(\chi^{2^{j-1}}))$ as $\Frob^\ast W^{(j-1)}\simeq W^{(j+1)}$.
\end{proof}
\begin{lemma}\label{Lem:IsPol}
  The representations $L (\lambda | \chi^{\mu})$ from Theorem~\ref{thm:classif rep} are polynomial.
\end{lemma}

\begin{proof}
  Let $X = m + nP$. Then $L (\lambda | \1)$ is polynomial, as,
  via the usual construction for $\operatorname{GL} (m)$-representations, it is a
  subquotient of $\left( \bigwedge^m X \right)^{\otimes \lambda_m} \otimes
  \left( \bigwedge^{m - 1} X \right)^{\otimes \lambda_m - \lambda_{m - 1}}
  \otimes \cdots$, thus a subquotient of a tensor power of $X$.
  
  Let $L_{\chi} = L (0 | \chi)$. Then $L (0
  | \chi^{\mu})$ is a subquotient of $$\left( \bigwedge^n L_{\chi}
  \right)^{\otimes \mu_n} \otimes \left( \bigwedge^{n - 1} L_{\chi}
  \right)^{\otimes \mu_n - \mu_{n - 1}} \otimes \cdots $$ so it suffices to
  show that $L_{\chi}$ is polynomial. But it is a subquotient of $S^4 (X)$, as described in \ref{lem:w2_irrep}.

  Then $L (\lambda | \chi^{\mu})$ is a subquotient of $L (\lambda
  | \1) \otimes L (0 | \chi^{\mu})$, so it is also
  polynomial.
\end{proof}

\begin{lemma}\label{Lem:notneg}
  $L (\lambda | \chi^{\mu})$ is not polynomial if $\lambda_i < 0$ or
  $\mu_i < 0$ for some $i$.
\end{lemma}

\begin{proof}
  If $\lambda_i < 0$, restrict $L$ to a representation of $\operatorname{GL} (m)
  \times \mathbb{G}_m^n$; if $\mu_i < 0$, restrict $L$ to a representation of
  $\mathbb{G}_m^m \times \operatorname{GL} (n)$. Then $L$ has some non-polynomial
  $\operatorname{GL} (m)$ (resp. $\operatorname{GL} (n)$) weight, so it is not polynomial.
\end{proof}

\begin{lemma}
  The $\operatorname{GL} (1 + nP)$-irrep $L \assign L (0 | \Lambda_1, \ldots, \Lambda_n)$
  is polynomial if and only if each $L_i$ is a non-negative power of $\chi$. 
\end{lemma}

\begin{proof}
  We induct on $n$. As the base case, suppose $n = 1$ and $L_1 \neq
  \chi^{\ell}$. We know the $\operatorname{GL} (1 + P)$-irrep of highest weight $(0,
  L_1)$ is then not polynomial (it contains a weight space $(- 1| T)$, as described in \eqref{eq:r_def}). Furthermore, if $L_1 =
  \chi^{\ell}$ for $\ell<0$, it is not polynomial by Lemma~\ref{Lem:notneg}. The conclusion thus follows from Lemma~\ref{Lem:IsPol}
  
  Now suppose we know the statement is true for $\operatorname{GL} (1 + (n - 1) P)$. Embed $\GL(1 + (n-1)P)$ in $\GL(1 + nP)$ via $M \mapsto \left(\begin{array}{cc}
    M & 0\\
    0 & I_P
  \end{array}\right)$ with $M \in \operatorname{GL} (1 + (n - 1) P)$,  $\operatorname{Res}_{\operatorname{GL} (1
  + (n - 1) P)}^{\operatorname{GL} (1 + nP)} L$ contains a $\operatorname{GL} (1 + (n - 1)
  P)$-irrep of highest weight $0|\Lambda_1, \ldots, \Lambda_{n - 1}$, so we know that
  $\Lambda_1, \ldots, \Lambda_{n - 1}$ are powers of $\chi$. So it remains to show that
  $L_n = \chi^{\ell}$. Choosing a different Borel, consider $L$ as a
  representation of $\operatorname{GL} ((n - 1) P + 1 + P)$. The highest weight of $L$ with respect to this Borel can be determined by applying $R_{1, 1}, \dots, R_{1, n - 1}$; because each $\Lambda_i$ is a power of $\chi$ and $R(0|\chi^\ell) = 0 | \chi^\ell$, the highest weight of
  $L$ is then still $0 | \Lambda$. However, then we may restrict $L$ to
  $\operatorname{GL} (1 + P) \subset \operatorname{GL} ((n - 1) P + 1 + P)$ to see that
  $\operatorname{Res}_{\operatorname{GL} (1 + P)}^{\operatorname{GL} (1 + nP)} L$ contains the
  $\operatorname{GL} (1 + P)$ irrep with highest weight $0|\Lambda_n$, which is not
  polynomial unless $L_n$ is a power of $\chi$.
\end{proof}

\begin{lemma}\label{Lem:NotPol2}
  If for $\Lambda_1, \ldots, \Lambda_n$ in $\bT$, there exists $1\le i\le n$ for which $\Lambda_i$ is not a non-negative power of $\chi$, then $L \assign L (\lambda
  | \Lambda_1, \ldots, \Lambda_n)$ is not polynomial for any $\lambda$.
\end{lemma}

\begin{proof}
  Since $m > d$, we must have $\lambda_m = 0$. Then
  $\operatorname{Res}_{\operatorname{GL}_1^{m - 1} \times \operatorname{GL} (1 + nP)}^{\operatorname{GL} (m +
  nP)} L$ contains the $\operatorname{GL}_1^{m - 1} \times \operatorname{GL} (1 + nP)$ irrep
  with highest weight $0 | \Lambda_1, \ldots, \Lambda_n$, so by the previous
  lemma, the $\Lambda_i$ are all powers of $\chi$.
\end{proof}

\begin{proof}[Proof of Theorem~\ref{thm:classif rep}]
    One direction of the claim follows from Lemma~\ref{Lem:IsPol}, the other from Lemmas~\ref{Lem:notneg} and~\ref{Lem:NotPol2}.
\end{proof}

\begin{theorem}\label{thm:classif:fun}
    We can label the indecomposable modules in $\bB^d[\Ver_4^+]$ as $Y^{\lambda|\mu}$, for all pairs of partitions $(\lambda, \mu)$ with $|\lambda| + 4|\mu| = d$, such that
    $$\mD_{Y^{\lambda|\mu}}(X)\;\simeq\; L(\lambda|\chi^\mu)$$
    as $\GL(X)$-representations, for all $X=m+nP\in\Ver_4^+$ with $m,n$ sufficiently large ($X$ relatively $d$-discerning).
\end{theorem}
\begin{proof}
    This follows immediately from Lemma~\ref{lem:poly_glx} and Theorem~\ref{thm:classif rep}.
\end{proof}

\begin{remark}
    \begin{enumerate}
        \item We have $Y^{
\lambda|0} = Y^\lambda$, the Young module corresponding to $\lambda$, and $Y^{0|4} = R$, the indecomposable $S_4$-module defined in Section 5.4.4 of \cite{coulembier_inductive_2024}.
\item Theorem~\ref{thm:classif:fun} shows that if $X = m + nP$ for $m > d$ and $n > d/4$, then $X$ is $d$-discerning, though this is not a necessary condition. 
\item Theorem~\ref{thm:classif:fun} is the analogue of Donkin's classification of signed Young modules in characteristic $p>2$ in \cite[\S 1]{Donkin}.
\item We will abbreviate $\mD_{Y^{\lambda|\mu}}$ to $\mD_{\lambda|\mu}$, hence
$\mD_{\lambda|\mu}(X)\;\simeq\; L(\lambda|\chi^\mu)$
for $X=m+nP$ with $m>d$.
    \end{enumerate}
\end{remark}

\subsection{Additive polynomial functors}
In this section we classify additive polynomial functors in $\Ver_4^+$. We start with the simple ones.

\begin{proposition}\label{prop:SimpFun}
  The simple additive polynomial functors on $\Ver_4^+$ are $$\Fr_{\mathrm{even}}^i:=\mfrak{D}_{(2^i)|(0)
  }, \quad \Fr_{\mathrm{even}}^j \circ \Fr_{\mathrm{odd}}=\mfrak{D}_{(0)|(2^{j})},\quad\mbox{for }\; i,j\in\mN.$$
\end{proposition}
\begin{remark}
    Note that $\Fr_{\mathrm{even}} = \Fr_+$, so $\Fr^i_{\mathrm{even}} = \Fr_+^i =\Fr_+^{(i)}$. 
\end{remark}

\begin{proof}
  The polynomial functor $\mfrak{D}_{(2^i)|(0)}$ is the $i$th Frobenius twist, acting as $V \oplus W \otimes P \mapsto V^{(i)}$. Meanwhile, by Lemma~\ref{lem:w2_irrep}, $\mfrak{D}_{(0)|(2^j)}$
  acts as $V \oplus W \otimes P \mapsto W^{(j + 2)}$, which
  is also additive. 
  
  To see that these are the only simple additive functors, note that we only have to
  check simple polynomial functors that do not vanish on $\1$ or $P$ of
  degree $2^i$. Since $\mfrak{D}_{\lambda|\mu}$ is a subquotient of $\mfrak{D}_{\lambda|(0)} \otimes \mfrak{D}_{(0)|\mu}$ and the
  latter kills $\1$ when $\mu$ is nontrivial, we have that only
  functors of the form $\mfrak{D}_{\lambda|(0)}$ do not vanish on
  $\1$. From the classical theory of polynomial functors, such
  functors are only additive if $\lambda = (2^i, 0, \ldots, 0)$.

  It remains to check the simple polynomial functors that do not send $P$ to 0, thus producing the simple
  polynomial representations of degree $2^i$ of $\operatorname{GL} (P)$. We can enumerate those via the corresponding representations, labelled by $\bT$, thanks to Lemma~\ref{lem:poly_glx}.
  \begin{itemize}
    \item If $i = 1$, in degree 2 there is only one $\GL(P)$-representation, $T_2$, which thus must be either the evaluation of $\mfrak{D}_{(2)|(0)}$ or $\mfrak{D}_{(1, 1)|(0)}$. The former is the
    Frobenius twist, so it vanishes on $P$, and the latter is not additive
    when applied to $\1$.
    
    \item If $i \ge 2$, there are two polynomial representations, $\chi^{2^{i - 1}}$ and
    $\xi \chi^{2^{i - 1}}$. It is clear that $\chi^{2^{i - 1}}$ is the evaluation of $\mfrak{D}_{(0)|(2^{i-1})}$, so it is additive. Meanwhile, there exists a degree 4 partition $\lambda$ such that $\xi \chi = \mD_{(\lambda)|(0)}(P)$. Hence, $\xi \chi^{2^{i - 1}}$ must be the evaluation of $\mfrak{D}_{\lambda|(2^{i - 1} - 1)}$ (even of $\mfrak{D}_{\lambda|0}\otimes\mfrak{D}_{(2^{i - 1} - 1)}$). But this cannot be
    additive, because we know it kills $\1$: if it were additive,
    then $T (m + P) = T (P)$ for all $m$, but clearly this is not true for $m
    \geqslant \ell(\lambda)$, as $T$ will contain the $\operatorname{GL} (m)$ representation with
    highest weight $\lambda$. 

  \end{itemize}
  Therefore, the only simple additive polynomial functors that do not vanish
  on $P$ are those corresponding to $\lambda = 0$ and $\mu = (2^j, 0, \dots, 0)$.
\end{proof}

\begin{remark}
  If the Steinberg tensor product theorem (\cite{hu_supergroups_2024} 4.2.5) were true, then we could say
  immediately that $L (\lambda | \chi^{\mu}) = L (\lambda |
  \1) \otimes L (0 | \chi^{\mu})$, and we would be done
  because tensor products of polynomial functors cannot be additive.
\end{remark}

By Proposition~\ref{prop:add}, classifying all additive polynomial functors of degree $d$, say for $d=2^j$, $j>1$, now corresponds to classifying all polynomial functors for which the only simple constituents are $\mD_{(2^j)|(0)}$ and $\mD_{(0)|(2^{j-2})}$.
We realise this by applying Theorem \ref{thm:cd_comodules}, so we can describe the additive polynomial functors in degree $d$ by choosing $X$ relatively $d$-discerning, finding $C_X^d$, and describing the $C_X^d$-comodules compatible with $\varepsilon$. Let $X = m + nP$ for $m, n > d$; then $C_X^d$ is the subalgebra of `$\Delta_+$-primitive elements' of $k[\End(X)]$, meaning elements that satisfy $\Delta_+(x)=x\otimes 1+1\otimes x$.

\begin{proposition}\label{prop:describeC}
    If $d$ is not a power of 2, $C_X^d = 0$. If $d = 1$, $C_X^1$ consists of all the degree 1 elements of $k[\End(X)]$. If $d = 2$, $C_X^2 = \text{span}(F_{ij}^2)$. If $d = 2^\ell$ for $\ell > 1$, $C_X^d = \text{span}(F_{ij}^d, B_{ij}^d, C_{ij}^d, D_{ij}^d, E_{ij}^d)$.
\end{proposition}
\begin{proof}
    All the degree 1 elements of $k[\End(X)]$ are primitive with respect to $\Delta_+$. Note that if $y$ is primitive, we have $$\Delta_+(y^2) = y^2 \otimes 1 + 1 \otimes y^2 + y' \otimes y'.$$ Therefore, if $w \in \ker \delta$ is primitive, so is $w^2$, which implies that $\Delta_+(w^{2^m}) = w^{2^m} \otimes 1 + 1 \otimes w^{2^m}$ for all $m$. We can also then check that if $y \in k[\End(X)]$ is primitive, so is $y^4$. Hence, $\Delta_+(y^{2^n}) = y^{2^n} \otimes 1 + 1 \otimes y^{2^n}$ when $n \ge 2$ for all $y$.
    
    Therefore, $F_{ij}^2$ and $F_{ij}^d, B_{ij}^d, C_{ij}^d, D_{ij}^d, E_{ij}^d$, $d = 2^\ell, \ell > 1$ are primitive. We claim these form a basis for the $\Delta_+$-primitive elements in degree $> 1$.
    
    We may interpret elements of $\End(X)(A)$ as matrices and elements of $k[\End(X)]$ as (noncommutative) polynomials $f(X_1, \dots, X_N)$ on the entries of such matrices $X_1, \dots, X_N$. Then primitive elements in $k[\End(X)]$ are additive polynomials, i.e. for two matrices $X, Y$, $f$ is additive if $$f(X_1 + Y_1, \dots, X_N + Y_N) = f(X_1, \dots, X_N) + f(Y_1, \dots, Y_N).$$ We claim that such polynomials only involve monomials of the form $(X_i)^{2^\ell}$. By using the braiding, we can make it so that $f$ only consists of monomials of the form $X_1^{e_1} \cdots X_n^{e_n}$. Since $f$ is additive, we must have that $$f(0, \dots, 0, X_i, 0, \dots, 0) + f(X_1, \dots, X_{i-1}, 0, X_{i+1}, \dots, X_n) = f,$$
    which implies that $f$ consists only of monomials in $X_i$ and monomials without $X_i$ at all. Applied to all $X_i$, this implies that $f$ consists only of monomials of the form $X_i^{e_i}$, so we can write $f = f_1(X_1) + f_2(X_2) + \cdots + f_N(X_N)$, with each $f_i$ additive. Now the $f_i$ must involve only exponents that are powers of 2. We must have $f_i(\alpha X_i) + f_i((1 + \alpha)X_i) = f_i(X_i)$ for all $\alpha \in k$, so for all $\alpha$, we must have $\alpha^{m} + (1+\alpha)^{m} = 1$ where $X_i^{m}$ has nonzero coefficient in $f_i$. But this is only possible if $m = 2^j$ for some $j$.
    
    However, if $y \notin \ker \delta$, we know that $y^2$ is not primitive with respect to $\Delta_+$. So only 4th powers and above will be primitive. 
\end{proof}

\begin{remark}\begin{enumerate}
    \item When $d \ne 1$, $C_X^d$ lives in $\Vecc \subset \Ver_4^+$. 
    \item A $C_X^d$-comodule $Y$ is compatible with $\varepsilon$, viewed as a map $C_X^d \to k[\pi_1]$, if the natural $k[\pi_1]$-comodule structure on $Y$ agrees with the one induced by $\varepsilon$. In particular, if $Y$ has nontrivial $\delta$-action and $d > 1$, it cannot be compatible with $\varepsilon$ since $\delta$ does not lie in the image of $\varepsilon: C_X^d \to k[\pi_1]$.
\end{enumerate}
    
\end{remark}

\begin{theorem}\label{thm:alladditive}
    \begin{enumerate}
        \item The only polynomial functors in degree 1 are direct sums of the identity and they are additive.
        \item The only additive polynomial functors in degree 2 are direct sums of $\Fr_+$.
        \item If $d = 2^\ell$ for $\ell \ge 2$, then the category of additive polynomial functors of degree $d$ has 5 indecomposable objects. The projective cover of $\mfrak{D}_{(d)|(0)}$ has length two with socle $\mfrak{D}_{(0)|(d/4)}$, the injective of $\mfrak{D}_{(d)|(0)}$ hull has length 2 with top $\mfrak{D}_{(0)|(d/4)}$. Finally the projective cover of $\mfrak{D}_{(0)|(d/4)}$ is also its injective hull and has length three with $\mD_{(d)|0}$ in the middle.
    \end{enumerate}
\end{theorem}
\begin{proof}
    Part (1) is a generality, see for instance \cite[Example~8.1.8]{coulembier_inductive_2024}. Alternatively we can use Proposition~\ref{prop:describeC} and observe that $\uEnd(X)$ is Morita equivalent with $\1$.
    Part (2) follows directly from Propositions~\ref{prop:SimpFun} and~\ref{prop:describeC}. 
    
    We make the claims in part (3) more explicit, in terms of $C_X^d$, for $X=m\1+nP$. We already know via Proposition~\ref{prop:SimpFun} that $C_X^d$ has two simple comodules compatible with $\varepsilon$: $M_1^d$, which has underlying object $n\1$ and corresponds to $\mfrak{D}_{(0)|(2^{\ell - 2})}$, and $M_2^d$, which corresponds to $\mfrak{D}_{(d)|(0)}$ and has underlying object $m\1$. We claim that the injective hull of $\mfrak{D}_{(0)|(2^{\ell - 2})}$ is isomorphic to the span of $D_{1j}^d, C_{1j}^d$ and is an extension of $\mfrak{D}_{(0)|(2^{\ell - 2})}$ by $\mfrak{D}_{(d)|(0)}$, while the injective hull of $\mfrak{D}_{(d)|(0)}$ is isomorphic to the span of $F_{1j}^d$, $B_{1j}^d$, $E_{1j}^d$ and has composition series $\mfrak{D}_{(d)|(0)}, \mfrak{D}_{(0)|(2^{\ell - 2})}, \mfrak{D}_{(d)|(0)}$.

    To prove this we start from 
    \begin{align*}
        \Delta_\times(F_{ij}^d) &= \sum_{k = 1}^m F_{ik}^d \otimes F_{kj}^d \\
        \Delta_\times(B_{ij}^d) &= \sum_{k = 1}^m B_{ik}^d \otimes F_{kj}^d + \sum_{k = 1}^n D_{ik}^d \otimes B_{kj}^d \\
        \Delta_\times(C_{ij}^d) &= \sum_{k = 1}^m F_{ik}^d \otimes C_{kj}^d + \sum_{k = 1}^n C_{ik}^d \otimes D_{kj}^d \\
        \Delta_\times(D_{ij}^d) &= \sum_{k= 1}^n D_{ik}^d \otimes D_{kj}^d \\
        \Delta_\times(E_{ij}^d) &= \sum_{k = 1}^m B_{ik}^d \otimes C_{kj}^d + \sum_{k = 1}^n (E_{ik}^d \otimes D_{kj}^d + D_{ik}^d \otimes E_{kj}^d).
    \end{align*}
    Since $C_X^d$ acts on itself via $\Delta_\times$, for fixed $1 \le i \le m$, the span of $\{F_{ij}^d, 1 \le j \le m\}$ is a simple $C_X^d$-comodule corresponding to the $\GL(X)$-module where $(F, B, C, D, E)$ acts as $F^{(\ell)}$. Likewise, for fixed $1 \le i \le n$, the span of $\{D_{ij}^d, 1 \le j \le n\}$ is a simple $C_X^d$ comodule corresponding to the $\GL(X)$ module where $(F, B, C, D, E)$ acts as $D^{(\ell)}$. We also see that for fixed $1 \le i \le m$, the span of $F_{ij}^d, C_{ij}^d$ is indecomposable and an extension of $\mfrak{D}_{(0)|(2^{\ell - 2})}$ by $\mfrak{D}_{(d)|(0)}$, corresponding to the $\GL(X)$ module where $(F, B, C, D, E)$ acts as $\begin{pmatrix} F^{\ell} & C^{\ell} \\ 0 & D^{\ell} \end{pmatrix}$. Likewise, for fixed $1 \le i \le n$, the span of $B_{ij}^d$, $D_{ij}^d$, and $E_{ij}^d$ is indecomposable and has composition series $\mfrak{D}_{(d)|(0)}, \mfrak{D}_{(0)|(2^{\ell - 2})}, \mfrak{D}_{(d)|(0)}$, corresponding to the $\GL(X)$ module where $(F, B, C, D, E)$ cats as $\begin{pmatrix} D^{(\ell)} & B^{(\ell)} & E^{(\ell)} \\ 0 & F^{(\ell)} & C^{(\ell)} \\ 0 & 0 & D^{(\ell)} \end{pmatrix}$. By looking at $\Delta_\times$, we see these are all the direct summands of $C_X^d$.
\end{proof}

\begin{remark}
    Let $d = 2^\ell$ for $\ell \ge 2$. It follows that $\mfrak{D}_{(d)|(0)}$ and $\mfrak{D}_{(0)|(2^{\ell - 2})}$ have no self-extensions (since the additive polynomial functors form a Serre subcategory). This implies that over $\GL(m+nP)$, $L(0|\chi)$ has no self-extensions inside $\Rep (\GL(m+nP),\varepsilon)$ when $m\ge 2$ and $n\ge 1$ (using Proposition~\ref{prop:discerning} below). On the other hand, when $m=0$ and $n=1$, there is a self-extension of $L(0|\chi)$ by $L(0|\chi)$, which appears as a subquotient in $T_1 T_3$: it sends $\begin{pmatrix} a & a' \\ b & a + b' \end{pmatrix} \in \GL(P)(A)$ to $$\begin{pmatrix} a^4 & a^2b^2 + a^2ba' + a^3b' + aba'b' \\ 0 & a^4 \end{pmatrix}.$$
\end{remark}

\begin{proposition}
    Ignoring degree 1, there are only exact polynomial functors in degree $2^\ell$ for $\ell\ge 2$. In the latter degrees the unique indecomposable exact polynomial functor corresponds to the projective injective object in the category of additive polynomial functors.
\end{proposition}
\begin{proof}
    If $F$ is an exact polynomial functor, then the length of $\mF(P)$ must be twice the length of $\mF(\unit)$. By the classification of additive polynomial functors in Theorem~\ref{thm:alladditive}, the only candidate for an exact polynomial functor is thus the projective polynomial functor of length 3 in degree $2^\ell$ for $\ell\ge 2$. We argue that these polynomial functors, which we label $F_{2^\ell}$, are indeed exact indirectly by showing that there must be some exact polynomial functors in those degrees.

    In \cite{CF}, for every tensor category $\cC$, an additive monoidal functor
    \begin{equation}\label{eqCF}\cC\to\cC\boxtimes\Ver_4\end{equation}
    is constructed which is polynomial of degree 4 (more precisely, such a functor $\Phi^\lambda_1$ is constructed for every choice $\lambda\in \mP^1(k)\backslash \mP^1(\mathbb{F}_2)$). Indeed, it is the composition of the fourth tensor power, as a functor from $\cC$ to $\cC\boxtimes \Rep S_4$, follow by applying a (monoidal) functor $\Rep S_4\to \Ver_4$. Composing with an arbitrary faithful exact functor $\Ver_4\to\Vecc$ yields by definition the evaluation of a degree 4 polynomial functor for $\cC$.
    
    Moreover, by 7.2.1 and 7.2.4 in \cite{CF}, for $\lambda\in\mP^1(\mathbb{F}_4)\backslash\mP^1(\mathbb{F}_2)$, the functor \eqref{eqCF} is exact for $\cC=\Ver_4^+$ (even for $\Ver_4$) and moreover, takes values in $\Vecc\boxtimes\Ver_4\subset\Ver_4^+\boxtimes\Ver_4$. Hence, it must be the case that $F_{4}$ is exact, and moreover, that $\mF_4:\Ver_4^+\to\Ver_4^+$ takes values in $\Vecc$. The latter implies that there are exact polynomial functors $\Fr^{i}\circ \mF_4$ in degrees $2^{i+2}$ for all $i\in\mathbb{N}$, which forces $F_{2^{i+2}}$ to be exact.
\end{proof}

\subsection{Simple polynomial representations}

As observed in \cite{Brundan-Kujawa}, the Mullineux involution (relating the label of a simple $S_n$-representation to the label of the simple representation obtained by taking the tensor product with the sign representation in characteristic $p>2$) can be computed using odd reflections of supergroups. Moreover, these odd reflections are useful for classifying polynomial representations. In this section we perform the analogue for $p=2$, where no actual Mullineux involution exists. This allows us to determine how simple polynomial functors evaluate on arbitrary objects in $\Ver_4^+$, which is the main result of this subsection. In other words, we classify simple polynomial representations of all $\GL(m+nP)$ (now without restriction on $m$) by labelling them according to which polynomial functor they come from.
This in turn allows us to determine which objects are relatively $d$-discerning and $d$-faithful.

\begin{definition}
    Let $\lambda$ be any partition. Then there is a unique pair of a 2-restricted partition, which we denote $\overline{\lambda}$, and an even partition, which we denote $\lambda^*$, such that $\overline{\lambda} + \lambda^* = \lambda$. 

For $m \in \mbb{N}$, we define $\lambda[m] := (\lambda_{m + 1}, \lambda_{m + 2}, \dots )$, e.g. $\lambda[0] = \lambda$. We use similar notation for $\GL(nP)$-highest weights.
\end{definition}

\begin{definition}
    Let $\lambda$ be a 2-restricted partition. The rim of $\lambda$ consists of all nodes $(i, j) \in \lambda$ such that $(i + 1, j + 1) \notin \lambda$. Define $j(\lambda)$ to be the union of 4-segments, which are subsets of $\lambda$ containing up to 4 nodes, as follows, starting at the left-most column:
    \begin{itemize}
        \item if the starting column of the segment is not the last column, the 4-segment is the bottommost 4 nodes of the rim in this column. If the rim has less than 4 nodes in this column, the 4-segment is the bottommost 2 nodes from this column and up to 2 nodes from the column to the right.
        \item If the starting column is the last column, the bottommost up to 4 nodes of the rim in this column.
    \end{itemize}

    Denote by $0\le j(\lambda)_i\le 2$ the number of boxes in row $i$ of $j(\lambda)$. 
\end{definition}
Note that, in all but the last column, $j(\lambda)$ has an even number of nodes. Moreover, all but the last 4-segment will contain exactly 4 nodes, and the last will contain at most 4 nodes.

\begin{proposition}\label{prop:j}
    For a 2-restricted partition $\lambda$, the set $j(\lambda)$ has the following properties:
    \begin{enumerate}
        \item $j(\lambda)$ is a nonempty subset of the rim;
        \item $\lambda \backslash j(\lambda)$ is also 2-restricted.
    \end{enumerate}
\end{proposition}
\begin{proof}
    The first is immediate. For the second, note that $(j(\lambda)_i, j(\lambda)_{i + 1})$ is one of $(2, 2)$, $(1, 1)$, $(0, 0)$, $(0, 1)$, $(1, 2)$, $(2, 1)$, or $(1, 0)$. The possibility $(0, 1)$ and $(1, 2)$ occurs only when $\lambda_i = \lambda_{i + 1}$, and the possibilities $(2, 1)$ and $(1, 0)$ occur only when $\lambda_i = \lambda_{i + 1} + 1$. Therefore, if $\lambda_i - \lambda_{i + 1} \in \{0, 1\}$, we will still have $$(\lambda_i - j(\lambda)_i) - (\lambda_{i + 1} - j(\lambda)_{i + 1}) \in \{0, 1\}.$$ So $\lambda \backslash j(\lambda)$ is 2-restricted.
\end{proof}

\begin{definition}
For 2-restricted $\lambda$, let $J^0(\lambda) = \lambda$ and $J^{i}(\lambda) = J^{i-1}(\lambda) \backslash j(J^{i-1}(\lambda))$; this is well-defined as by Proposition~\ref{prop:j}, each $J^{i}(\lambda)$ is 2-restricted and $J^i(\lambda)$.  We also write $J(\lambda)=J^1(\lambda)$.
\end{definition}
Note that $J^i(\lambda)=\varnothing$ for $i\gg 0$ by Proposition~\ref{prop:j}. 

\begin{example}
    Suppose $\lambda = (4, 3, 3, 3, 2, 2, 2, 1, 1)$. Then $J(\lambda) = 2, 2, 2, 2, 2, 1, 1$. The nodes highlighted in yellow are $j(\lambda)$, and the remaining partition is $J(\lambda)$.

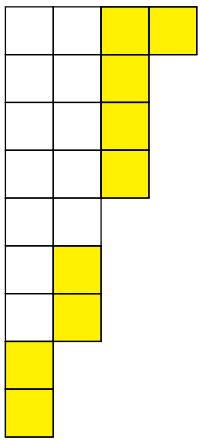
\begin{center}
    \ydiagram{4, 3, 3, 3, 2, 2, 2, 1, 1}
    *[*(yellow)]{2 + 2, 2 + 1, 2 + 1, 2 + 1, 0, 1 + 1, 1 + 1, 1, 1}
\end{center}
\end{example}

\begin{example}
    If $\lambda = (1^d)$ and $d > 4$, $j(\lambda)$ is the 4 bottom nodes and $J(\lambda) = (1^{d - 4})$. If $d \le 4$, then $j(\lambda) = \lambda$ and $J(\lambda) = (0)$.
\end{example}

\begin{example}
    If $\lambda = (2, 1^d)$ and $d \ge 3$, then $j(\lambda)$ is the 4 bottom nodes of the first column and the only node in the second column. Then $J(\lambda) = (1^{d - 4})$ if $d \ge 4$ and $(0)$ if $d = 3$. If $d = 1$ or $2$, then $j(\lambda)$ is the bottom 2 nodes of the first column and the only node in the second column; hence, $J(\lambda) = (0)$ when $d = 1$ and $(1)$ when $d = 2$.
\end{example}

\begin{definition}
Call a 2-restricted partition $\lambda$ oddly regular if there does not exist $i$ such that $\lambda_i = \lambda_{i + 1} = \lambda_{i + 2}$ are all odd. 
\end{definition}

\begin{example}
Recall that the 2-rim of $\lambda$ is a subset of the rim, defined as the union of 2-segments as follows: the first 2-segment is the first $2$ nodes of the rim from left to right, and the remaining 2-segments are inductively defined as the next up to $2$ nodes on the rim (reading from left to right) starting from the column to the right of the previous 2-segment. If $\lambda$ is 2-restricted, the 2-rim can also be described as the bottommost up to 2 nodes of the rim in each column.

Hence, if $\lambda$ is oddly regular, then $j(\lambda)$ is the 2-rim of $\lambda$. 
\end{example}

\begin{definition}
    Suppose $\lambda$ is 2-restricted. Define the $\GL(nP)$-weight $M(\lambda)$ by
    \begin{itemize}
    \item Suppose $|j(\lambda)| = 4q + r$, $0 \le r < 4$. Then $$M(\lambda)_1 = \chi^q \begin{cases} T_2 & r = 2 \\ \xi T_1  & r = 1, \lambda_1 \text{ even} \\ T_1  & r = 1, \lambda_1 \text{ odd} \\ \xi T_3  & r = 3, \lambda_1 \text{ even} \\ T_3  & r = 3, \lambda_1 \text{ odd} \\ \1  & r = 0, \lambda_1 \text{ even} \\ \xi  & r = 0, \lambda_1 \text{ odd,} \end{cases}$$.
    \item $M(\lambda)[1] = M(J^1(\lambda))$.
\end{itemize}
\end{definition}

\begin{example}
    If $\lambda=(3,2,1)$, then $J(\lambda)=1$ and $M(\lambda)= (\chi T_1, T_1)$.
\end{example}

\begin{proposition}

Let $L$ be the irreducible $\GL(m + nP)$ representation with highest weight $\lambda|\1$ for $n \gg 0$. Then
$$0|M(\lambda) = (R_{1, n} \cdots R_{m, n})\cdots(R_{1, 2} \cdots R_{m, 2})(R_{1, 1} \cdots R_{m, 1})(\lambda|\1), $$
i.e. $0|M(\lambda)$ is the highest weight of $L$ with respect to the Borel $B(nP + m)$.
\end{proposition}
\begin{proof}
    Let $\Lambda_0 = \1$; define $\Lambda_i$ and $\mu_1, \dots, \mu_m$ via $R(\lambda_{m-i}|\Lambda_{i-1}) = \mu_{m - i}|\Lambda_i$. Therefore, we have
    $$(R_{i, 1} \cdots R_{m, 1})(\lambda|\1) = \lambda_1, \dots, \lambda_{i - 1}, \mu[i - 1]|\Lambda_{m - i + 1}, \1, \dots, \1$$
    and in particular
    $$(R_{1, 1} \cdots R_{m, 1})(\lambda|\1) = \mu|\Lambda_m, \1, \dots, \1.$$

    So we need to show that $\Lambda_m = M(\lambda)_1$ is as described in the proposition statement, and that $\mu = J(\lambda)$.

    We induct on the length of $\lambda$. If $\ell(\lambda) = 0$, then $|j(\lambda)| = \lambda_1 = 0$ and $M(\lambda) = \1$. If $\ell(\lambda) = 1$, then $|j(\lambda)| = \lambda_1 = 1$, and $M(\lambda) = T_1$. Suppose now that the statement is true for $\lambda[1]$, which is also 2-restricted; therefore, $\Lambda_{m - 1} = M(\lambda[1])_1$ and is dependent on $j(\lambda[1])$ and $\lambda_2$ as stated in the proposition. We are then interested in what $j(\lambda)$ is. Notice that $j(\lambda)$ depends on $j(\lambda[1])$ and $\lambda_1 - \lambda_2$. 
    
    If $\lambda_1 = \lambda_2$, then $\lambda$ and $\lambda[1]$ have the same number of columns. So $j(\lambda)$ includes the last box in the first row of $\lambda$ unless the last 4-segment already has 4 boxes. Hence, $j(\lambda) = j(\lambda[1])$ in this case if $|j(\lambda[1])| \equiv 0 \pmod{4}$; otherwise, $j(\lambda)_1 = 1$.

    If $\lambda_1 = \lambda_2 + 1$, then we have two cases. If the last 4-segment of $j(\lambda[1])$ contains 2 or 4 boxes (equivalently, $|j(\lambda)|$ is even) from the second-to-last column of $\lambda$, then $j(\lambda)_1 = 1$. If the last 4-segment of $j(\lambda[1])$ contains 1 or 3 boxes from the second-to-last column of $\lambda$, then $j(\lambda)_1 = 2$.

    So we can work out what $\Lambda_m$ and $\mu_1$ must be in each case. Below, we write $\Lambda_{m - 1}$ modulo powers of $\chi$ for brevity, and it is easily checked that the powers of $\chi$ agree with the proposition statement as well.

    $\bullet$ First suppose $\lambda_1 = \lambda_2 = 2q+1$. If $|j(\lambda[1])| \equiv 0 \pmod{4}$, then $|j(\lambda)| = |j(\lambda[1])|$. By the inductive hypothesis $\Lambda_{m - 1} = \xi$, and by \eqref{eq:r_def} we have $R(2q + 1|\xi) = 2q + 1|\xi$. Hence, $\Lambda_m = \xi$ and $\mu_1 = \lambda_1$, i.e. $j(\lambda)_1 = 0$, as desired. If $|j(\lambda[1]) \not\equiv 0 \pmod{4}$, $|j(\lambda)| = |j(\lambda[1])| + 1$. By the inductive hypothesis $\Lambda_{m - 1}$ is either $T_1$, $T_2$, or $T_3$. Then we have we have $R(2q+1|T_1) = 2q|T_2$, $R(2q+1|T_2) = 2q|T_3$, $R(2q+1|T_3) = 2q|\chi\xi$, so in each case, $\Lambda_m$ is as desired, and $\mu_1 = \lambda_1 - 1$, i.e. $j(\lambda)_1 = 1$, as desired. 

      $\bullet$ Now suppose $\lambda_1 = \lambda_2 = 2q$. If $|j(\lambda[1])| \not\equiv 0 \pmod{4}$, again $|j(\lambda)| = |j(\lambda[1])| + 1$. By the inductive hypothesis $\Lambda_{m - 1} = \xi T_1, T_2$ or $\xi T_3$. We then have $R(2q|\xi T_1) = 2q - 1 | T_2$, $R(2q|T_2) = 2q - 1 | \xi T_3$ and $R(2q|\xi T_3) = 2q - 1 | \chi$, so $\Lambda_m = T_2, \xi T_3, \chi$ respectively and $\mu_1 = \lambda_1 - 1$, i.e. $j(\lambda)_1 = 1$, as desired. If $|j(\lambda[1])| \equiv 0 \pmod{4}$, then $|j(\lambda)| = |j(\lambda[1])|.$. By induction $\Lambda_{m - 1} = \1$ and $R(2q|\1) = 2q|\1$, so $\Lambda_m = \1$ and $\mu_1 = \lambda_1$, i.e. $j(\lambda)_1 = 1$, as desired.
      
        $\bullet$ Now suppose $\lambda_1 = 2q + 2$ and $\lambda_2 = 2q + 1$. If $|j(\lambda[1])|$ is odd, then $|j(\lambda)| = |j(\lambda[1])| + 2$. By induction $\Lambda_{m - 1} = T_1, T_3$. Then $R(2q + 2|T_1) = 2q|\xi T_3$, $R(2q + 2|T_3) = 2q | \chi \xi T_1$, so $\Lambda_m = \xi T_3, \chi \xi T_1$ respectively and $\mu_1 = \lambda_1 - 2$, as desired. If $|j(\lambda[1])|$ is even, then $|j(\lambda)| = |j(\lambda[1])| + 1$. By induction, $\Lambda_{m - 1} = \xi, T_2$. We have $R(2q + 2|\xi) = 2q + 1|\xi T_1$, $R(2q + 2|T_2) = 2q + 1|\xi T_3$, so $\Lambda_m = \xi T_1, \xi T_3$ respectively and $\mu_1 = \lambda_1 - 1$ as desired.

        $\bullet$ Finally, suppose $\lambda_1 = 2q + 1$ and $\lambda_2 = 2q$. If $|j(\lambda[1])|$ is odd, then $|j(\lambda)| = |j(\lambda[1])| + 2$. By induction $\Lambda_{m - 1} = \xi T_1, \xi T_3$. We have $R(2q + 1|\xi T_1) = 2q - 1|T_3$, and $R(2q + 1|\xi T_3) = 2q - 1|\chi T_1$, so $\Lambda_m = T_3, T_1$ respectively and $\mu_1 = \lambda_1 - 2$, as desired. If $|j(\lambda[1])|$ is even, then we have $|j(\lambda)| = |j(\lambda[1])| + 1$. By induction, $\Lambda_{m - 1} = \1, T_2$. We have $R(2q + 1|\1) = 2q|T_1$ and $R(2q + 1|T_2) = 2q|T_3$, so $\Lambda_m = T_1, T_3$ respectively and $\mu_1 = \lambda_1 - 1$ as desired.
\end{proof}



\begin{example}
    If $\lambda$ is oddly regular, then $j(\lambda)$ is the 2-rim of $\lambda$, and $J(\lambda) = \lambda[2]$. Then $|j(\lambda)| = \lambda_1 + \lambda_2$, and because $\lambda$ is 2-restricted, $j(\lambda)$ uniquely determines $(\lambda_1, \lambda_2)$. Hence, we can read off $M(\lambda)_i$ directly. Let $q_i$ be such that $\lambda_{2i - 1} + \lambda_{2i} = 4q_i + r$ for $0 \le r \le 3$. Then
    $$
    M(\lambda)_i = \chi^{q_i}  \begin{cases}
        T_1 & \lambda_{2i - 1} \text{ odd}, \lambda_{2i} \text{ even} \\
        T_2 & \lambda_{2i - 1} = \lambda_{2i} \text{ odd} \\
        \xi T_3 & \lambda_{2i - 1} \text{ even}, \lambda_{2i} \text{ odd} \\
        \1 & \lambda_{2i - 1} = \lambda_{2i} \text{ even}.
    \end{cases}
    $$
    In particular, $M(\lambda)$ has length $\lfloor \frac{\ell(\lambda) + 1}{2} \rfloor$.
\end{example}

\begin{example}\label{all_ones_m}
    If $\lambda = (1^{4q + r})$, then $j(\lambda)$ is the (up to) 4 bottom nodes and $J(\lambda) = (1^{4(q + 1) + r})$ if $q \ge 1$, and empty otherwise. Hence, $M(\lambda) = \chi^q T_r, \chi^{q - 1} \xi, \chi^{q - 2} \xi, \dots, \xi$, with $q + 1$ parts.
\end{example}

\begin{proposition}
    Let $m, n \gg 0$ and $L$ be the irreducible $\GL(m + nP)$ representation with highest weight $\lambda| \chi^\mu$, where $\lambda, \mu$ are any partitions. Then the highest weight of $L$ with respect to the Borel $B(nP + m)$ is $\lambda^*|\chi^\mu \cdot M(\bar{\lambda})$. Here $\chi^\mu \cdot M(\lambda)$ is the weight $\chi^{\mu_1} M(\lambda)_1, \dots, \chi^{\mu_n} M(\lambda)_n$.
\end{proposition}
\begin{proof}
    We can see this computationally by noting that for $\GL(1 + P)$, $L(2|\1)$ and $L(0|\chi)$ are characters of $\GL(1 + P)$. Hence, if $R(\alpha|T) = \beta|T'$ for $\alpha, \beta \in \mbb{Z}$ and $T, T'$ irreducible $\GL(P)$-representations, we have $R(2k + \alpha|\chi^\ell T) = 2k + \beta|\chi^\ell T'$.
\end{proof}

\begin{theorem}
    Let $\mfrak{D}_{\lambda|\mu}$ be a simple polynomial functor on $\Ver_4^+$. Then $\mfrak{D}_{\lambda| \mu}$ is nonzero on $m + nP$ if and only if 
    \begin{enumerate}
        \item $\lambda[m]$ is $2$-restricted, and
        \item $M(\lambda[m])$ has length at most $n$, and
        \item $\mu[n] = 0$ (i.e. $\mu$ has length at most $n$).
    \end{enumerate}
    If the conditions are satisfied then $\mD_{\lambda|\mu}(m+nP)=L(\lambda_1,\ldots,\lambda_m|\chi^\mu \cdot M(\lambda[m]))$.
\end{theorem}
\begin{proof}
    Let $M, N \gg 0$ and $L$ be the irreducible $\GL(M + NP)$-representation with highest weight $\lambda | \chi^\mu$. If $n \ge \ell(M(\overline{\lambda[m]}))$, then the highest weight of $L$ with respect to $$B := B(m + nP + (M - m) + (N - n)P)\qquad\mbox{is}$$
    $$
    \left(\lambda_1, \dots, \lambda_m, \lambda[m]^*\right) | M\left(\overline{\lambda[m]}\right) \cdot \chi^{(\mu_1, \dots, \mu_n)}, \chi^{\mu[n]}.
    $$

    If $n < \ell(M(\overline{\lambda[m]}))$, then $\mfrak{D}_{\lambda|\mu}$ will vanish on $m + nP$, as the highest weight of $L$ with respect to $B$ will still have nonzero $(M - m) + (N - n)P$-part. Applying \ref{rem:poly_vanishing}(2) with $X = m + nP$ and $Y = (M - m) + (N - n)P$, we therefore see that $\mfrak{D}_{\lambda, \mu}(X \oplus Y)$ does not have a component in degree $d=|\lambda|+4|\mu|$ with respect to $X$, thus $\mfrak{D}_{\lambda|\mu}(X) = 0$.
  \end{proof} 

\begin{corollary}
    The irreducible polynomial representations of $\GL(m + nP)$ are exactly those whose highest weights are of the form $(\lambda_1, \dots, \lambda_m) | M(\lambda[m]) \cdot \chi^\mu$ with $\lambda$ a partition such that $\lambda[m]$ is 2-restricted and $\ell(M(\lambda[m]), \ell(\mu) \le n$.
\end{corollary}

\begin{proposition}\label{prop:discerning}
Let $X = m + nP$. Then $X$ is relatively $d$-discerning iff
\begin{itemize}
    \item when $d = 1$, $m \ge 1$ or $n \ge 1$
    \item when $d = 2$, $m + n \ge 2$ and $m \ge 1$
    \item when $d = 3$, $m \ge 3$ or $n \ge 1, m \ge 1$
    \item when $d \ge 4$, $m \ge \lfloor d/2\rfloor$ and $n \ge \lfloor d/4 \rfloor$. 
\end{itemize}
\end{proposition}

\begin{proof}
If $d < 4$, we can simply check the partitions of $d$ to obtain the above conditions. For example, $\1+P$ is $3$-discerning simply because $\GL(\1+P)$ has three simple polynomial representations of degree 3, see Proposition~\ref{prop:GL1P}. So suppose $d \ge 4$.

The conditions $m \ge \lfloor d/2\rfloor$ and $n \ge \lfloor d/4 \rfloor$ are necessary. Indeed, if $d$ is odd, consider $\lambda = (3, 2, \dots, 2)\vdash d$, and if $d$ is even, consider $\lambda = (2, \dots, 2)\vdash d$, both of which are not 2-restricted and therefore vanish on $m + nP$ if $m <\ell(\lambda)= \lfloor d/2 \rfloor$. Likewise, for $d = 4a + r$ with $0\le r<4$, consider the highest weight $r|(\chi, \dots, \chi)$, whose corresponding polynomial functor will vanish on $m + nP$ when $n < a=\lfloor d / 4 \rfloor$.

To see the conditions sufficient, suppose we have partitions $\lambda, \mu$ with $|\lambda| + 4|\mu| = d$. If the smallest nonzero part of $\lambda$ is 2, we have $\ell(\lambda) \le \lfloor d/2\rfloor$ and $\ell(\mu) \le \lfloor d/4 \rfloor$, so that $\mD_{\lambda|\mu}(X)\not=0$ is immediate. Otherwise, suppose that $\lambda$ has $a_1>0$ components equal to $1$. Let $\nu = (\lambda_1, \dots, \lambda_{\ell(\lambda) - a_1})$, so $\nu$ has smallest part $2$ with $|\nu| = |\lambda| - a_1$ and $\ell(\nu) \le \frac{d - a_1}{2}$. Let $\Lambda = M(1, \dots, 1) \cdot \chi^\mu$, i.e. $\Lambda_j = M(1, \dots, 1)_j \chi^{\mu_j}$. Applying $R_{\ell(\lambda),1}, \dots, R_{\ell(\lambda) - a_1, 1}$ to $\lambda|\chi^{\mu}$, we get the weight $\nu|\Lambda$. Because $\ell(M((1, \dots, 1))) = \lceil a_1 / 4 \rceil$ (\ref{all_ones_m}), as long as $\lceil a_1 / 4 \rceil \le \lfloor d / 4 \rfloor$, we would have $\ell(\nu) \le \lfloor d / 2 \rfloor$ and $\ell(\Lambda) \le \lfloor d / 4 \rfloor$ as desired.

However, it is possible that $\lceil a_1 / 4 \rceil > \lfloor d / 4 \rfloor$ if $a_1 \ge d - 2$; we can have either $a_1 = d - 2$ or $a_1 = d$. If $a_1 = d - 2$, reflect only $a_1 - 1$ ones; then $\lceil a_1 - 1/ 4 \rceil = \lfloor d / 4 \rfloor$ while we'll have $\nu = (2, 1)$, and $d \ge 4$, so the functor will not vanish on $m + nP$. If $a_1 = d = 4b + r$, then reflect $a_1 - r$ ones; since $d \ge 4$, we will always have that $r \le \lfloor d / 2\rfloor$, so the functor will not vanish on $m + nP$.
\end{proof}

\begin{lemma}\label{Lem:lengthM}
    Let $\mu$ be a 2-restricted partition. Then $M(\mu)$ has at most $\lceil \frac{|\mu| + 1}{4} \rceil$ parts.
\end{lemma}
\begin{proof}
    Notice that for any 2-restricted partition $\mu$, if $|\mu| > 4$, then $|j(\mu)| \ge 4$ also; if $|\mu| = 4$, then $|j(\mu)| = 3$ or $4$, and if $|\mu| < 4$, then $j(\mu) = \mu$. Hence, if $|\mu| \equiv 0 \pmod{4}$, $M(\mu)$ has at most $|\mu| / 4 + 1$ parts, and otherwise, $M(\mu)$ has at most $\lceil |\mu| / 4 \rceil$ parts. In either case $M(\mu)$ has at most $\lceil \frac{|\mu| + 1}{4} \rceil$ parts.
\end{proof}

\begin{proposition}
    Let $X = m + nP$. Then $X$ is $d$-faithful, i.e.
$kS_d\to\End(X^{\otimes d})$ is injective, or equivalently
 $\mD_{\lambda|0}$ does not vanish on $X$ when $\lambda \vdash d$ is 2-restricted, if and only if  both $m + 4n \ge d$ and $m + 2n \ge 1 + \lfloor d / 2 \rfloor$.
    
\end{proposition}
\begin{proof}
    We first show that $m + 4n \ge d$ and $m + 2n \ge 1 + \lfloor d / 2 \rfloor$ are necessary conditions.
    
Consider the partition $\lambda = (1^d)$. If $\mD_{\lambda|0}$ does not vanishes on $m + nP$, then $$\mD_{\lambda|0}(m+nP)=L((1^m)|\Lambda),$$ for $\Lambda=M(1^{d-m})$ of length at most $n$. But $\Lambda=M(1^{d-m})$ has $\lceil (d-m)/4 \rceil$ parts, so we require $n\le \lceil (d-m)/4 \rceil$, which is equivalent to $4n\le d-m$.

    Likewise, consider the partition $\lambda = (2, 2, \dots, 2, 1)$ if $d$ is odd and $\lambda = (2, 2, \dots, 2, 1, 1)$ is $d$ is even. Then $\lambda$ has $r := 1 + \lfloor d / 2 \rfloor$ rows and $M(\lambda)$ has $\lceil r /2 \rceil$ parts, since $\lambda \backslash j(\lambda) = \lambda[2]$. As in the previous paragraph, it follows that $\mD_\lambda(m+nP)\not=0$ requires
    $$\lceil (1+\lfloor d/2\rfloor -m)/2\rceil\le n.$$

    Now we show these are sufficient. Let $\lambda \vdash d$ be an arbitrary partition and $m, n$ satisfying the above. If $\lambda_m = 1$, then  $M(\lambda[m])=M(1^{\ell})$ for $\ell\le d-m$, has at most $\lceil \frac{d - m}{4} \rceil$ parts. Because $m + 4n \ge d$, we have $n \ge \frac{d - m}{4}$ and $n$ is an integer, so $n \ge \lceil \frac{d - m}{4} \rceil$. Hence $\mD_{\lambda|0}$ is nonzero on $X$.
    If $m = 0$ or $\lambda_m \ge 2$, then $|\lambda[m]| \le d - 2m$. By Lemma~\ref{Lem:lengthM}, $M(\lambda[m])$ has at most $\lceil \frac{d - 2m + 1}{4} \rceil$ parts. But by assumption $$n \ge \frac{1 + \lfloor d / 2 \rfloor - m}{2} \ge \frac{2 + (d - 1) - 2m}{4}$$ as $2\lfloor d / 2 \rfloor \ge d - 1$. Since $n$ is an integer, $n \ge \lceil \frac{d - 2m + 1}{4} \rceil$, so $\mD_{\lambda|0}$ is nonzero on $X$.
\end{proof}

\begin{example}In this example, for $\mu$ the empty partition, we abbreviate $\mD_{\lambda|\mu}$ to $\mD_\lambda$.
    For any object $X$ in $\Ver_4^+$ (more generally, up to notation, in any symmetric monoidal tensor category over $k$), inside $\Rep \GL(X)\boxtimes\Rep S_3$, we have
    \[X^{\otimes 3}\;\simeq\; \mD_{2,1}(X)\boxtimes D^{(2,1)}\;\oplus \; [\mD_{(1,1,1)}(X),\mD_3(X),\mD_{(1,1,1)}(X)],\]
    where $D^{(2,1)}$ is the non-trivial (projective) simple $kS_3$-module, and the right term is indecomposable with length at most three, with $\mD_{(1,1,1)}(X)$ in top and socle and $\mD_3(X)$ in the middle. The action of $S_3$, under the canonical isomorphism
    \[k S_3\;\simeq\; \End_{k}(D^{(2,1)})\times k[x]/x^2\]
    is such that $\End(D^{(2,1)})$ acts in the obvious way and $x$ maps top to socle of the length 3 module. Hence, as per general theory, $X$ is 3-faithful if and only if
    \[kS_3\;\to\; \End_{\GL(X)}(X^{\otimes 3})\]
    is an isomorphism if and only if neither $\mD_{(2,1)}(X)$ nor $\mD_{(1,1,1)}(X)$ are zero. For example $X=P$ is 3-faithful, but not 3-discernable since $D_3(P)=0$. Moreover, we thus find in $\Rep GL(P)$
    \[P^{\otimes 3}\;\simeq\; (\xi T_3)^{\oplus 2}\oplus [T_3,T_3],\]
    where the right term is a non-split self-extension of $T_3$. In particular, the extension of $\xi T_3$ and $T_3$ in \cite[Proposition~4.10]{hu_supergroups_2024} is split.
\end{example}

\printbibliography

@book{EGNO,
	address = {Providence, Rhode Island},
	series = {Mathematical surveys and monographs},
	title = {Tensor categories},
	isbn = {978-1-4704-2024-6},
	language = {en},
	number = {volume 205},
	publisher = {American Mathematical Society},
	editor = {Etingof, P. I. and Gelaki, Shlomo and Nikshych, Dmitri and Ostrik, Victor},
	year = {2015},
}

@article{benson_new_2021,
	title = {New incompressible symmetric tensor categories in positive characteristic},
	author = {Benson, Dave and Etingof, Pavel and Ostrik, Victor},
issn = {0012-7094},
journal = {Duke mathematical journal},
language = {eng},
number = {1},
volume = {172},
year = {2023},
eprint={2003.10499},
      archivePrefix={arXiv},
      primaryClass={math.RT},
doi = {10.1215/00127094-2022-0030},
}

@article {Incomp,
    AUTHOR = {Coulembier, Kevin and Etingof, Pavel and Ostrik, Victor},
     TITLE = {Incompressible tensor categories},
   JOURNAL = {Adv. Math.},
  FJOURNAL = {Advances in Mathematics},
    VOLUME = {457},
      YEAR = {2024},
     PAGES = {Paper No. 109935, 65},
      ISSN = {0001-8708,1090-2082},
   MRCLASS = {18M05 (14F08 16G10 18M25 20G05)},
  MRNUMBER = {4794815},
MRREVIEWER = {Luz\ Adriana\ Mej\'ia Casta\~no},
       DOI = {10.1016/j.aim.2024.109935},
       URL = {https://doi.org/10.1016/j.aim.2024.109935},
}

@article{coulembier_frobenius_2022,
author = {Kevin Coulembier and Pavel Etingof and Victor Ostrik and Alexander Kleshchev},
title = {On Frobenius exact symmetric tensor categories},
volume = {197},
journal = {Annals of Mathematics},
number = {3},
publisher = {Department of Mathematics of Princeton University},
pages = {1235 -- 1279},
year = {2023},
doi = {10.4007/annals.2023.197.3.5},
      eprint={2107.02372},
      archivePrefix={arXiv},
      primaryClass={math.RT}
}

@article{deligne_categories_1990,
	title = {Catégories tannakiennes},
	volume = {87},
	journal = {Grothendieck Festschrift vol II. Progress in Mathematics},
	author = {Deligne, Pierre},
	year = {1990},
	pages = {111--195},
doi = {10.1007/978-0-8176-4575-5_3},
}

@article{coulembier_monoidal_2021,
	title = {Monoidal abelian envelopes},
	volume = {157},
	doi = {10.1112/S0010437X21007399},
	number = {7},
	journal = {Compositio Mathematica},
	author = {Coulembier, Kevin},
	month = jul,
	year = {2021},
	pages = {1584--1609},
eprint = {2003.10105},
archivePrefix = {arXiv},
primaryClass = {math.CT},
}

@article {EO-finite,
    AUTHOR = {Etingof, Pavel and Ostrik, Viktor},
     TITLE = {Finite tensor categories},
   JOURNAL = {Mosc. Math. J.},
  FJOURNAL = {Moscow Mathematical Journal},
    VOLUME = {4},
      YEAR = {2004},
    NUMBER = {3},
     PAGES = {627--654, 782--783},
      ISSN = {1609-3321,1609-4514},
   MRCLASS = {18D10 (16W30)},
  MRNUMBER = {2119143},
MRREVIEWER = {Sebastian\ Burciu},
       DOI = {10.17323/1609-4514-2004-4-3-627-654},
       URL = {https://doi.org/10.17323/1609-4514-2004-4-3-627-654},
}

@article{friedlander_cohomology_1997,
author = {Friedlander, Eric M. and Suslin, Andrei},
address = {Heidelberg},
copyright = {Springer-Verlag Berlin Heidelberg 1997},
journal = {Inventiones mathematicae},
number = {2},
pages = {209-270},
publisher = {Springer Nature B.V},
title = {Cohomology of finite group schemes over a field},
volume = {127},
year = {1997},
doi = {10.1007/s002220050119},
}

@article{drupieski_cohomology_2016,
author = {Drupieski, Christopher M.},
copyright = {2015 Elsevier Inc.},
issn = {0001-8708},
journal = {Advances in mathematics (New York. 1965)},
pages = {1360-1432},
publisher = {Elsevier Inc},
title = {Cohomological finite-generation for finite supergroup schemes},
volume = {288},
year = {2016},
doi = {10.1016/j.aim.2015.11.017},
eprint={1408.5764},
archivePrefix={arXiv},
primaryClass = {math.RT},
}

@misc{coulembier_inductive_2024,
      title={Inductive systems of the symmetric group, polynomial functors and tensor categories}, 
      author={Kevin Coulembier},
      year={2024},
      eprint={2406.00892},
      archivePrefix={arXiv},
      primaryClass={math.RT}
}

@article{deligne_tensorielles_2002,
title={Categories tensorielles},
author = {Pierre Deligne},
year = {2002}, 
journal = {Moscow Mathematical Journal},
volume = {2},
pages = {227-248},

}

@misc{hu_supergroups_2024,
    title={Representation Theory of General Linear Supergroups in Characteristic 2}, 
      author={Serina Hu},
      year={2024},
      eprint={2406.10201},
      archivePrefix={arXiv},
      primaryClass={math.RT},
}

@article {CEKO,
    AUTHOR = {Kevin Coulembier and Pavel Etingof and Alexander Kleshchev and Victor Ostrik},
     TITLE = {Super invariant theory in positive characteristic},
   JOURNAL = {Eur. J. Math.},
  FJOURNAL = {European Journal of Mathematics},
    VOLUME = {9},
      YEAR = {2023},
    NUMBER = {4},
     PAGES = {Paper No. 94, 39},
      ISSN = {2199-675X,2199-6768},
   MRCLASS = {17B10 (16W22 20G05)},
  MRNUMBER = {4651949},
MRREVIEWER = {Sabino\ Di Trani},
       DOI = {10.1007/s40879-023-00688-z},
       URL = {https://doi.org/10.1007/s40879-023-00688-z},
}

@misc{CS25,
      title={Homogeneous spaces in tensor categories}, 
      author={Kevin Coulembier and Alexander Sherman},
      year={2025},
      eprint={2505.04848},
      archivePrefix={arXiv},
      primaryClass={math.AG},
      url={https://arxiv.org/abs/2505.04848}, 
}

@misc{CF,
      title={Towards higher Frobenius functors for symmetric tensor categories}, 
      author={Kevin Coulembier and Johannes Flake},
      year={2024},
      eprint={2405.19506},
      archivePrefix={arXiv},
      primaryClass={math.RT},
      url={https://arxiv.org/abs/2405.19506}, 
}

@article {Tann,
    AUTHOR = {Kevin Coulembier},
     TITLE = {Tannakian categories in positive characteristic},
   JOURNAL = {Duke Math. J.},
  FJOURNAL = {Duke Mathematical Journal},
    VOLUME = {169},
      YEAR = {2020},
    NUMBER = {16},
     PAGES = {3167--3219},
      ISSN = {0012-7094,1547-7398},
   MRCLASS = {18M15 (14L15 16D90 16T05 20C05)},
  MRNUMBER = {4167087},
MRREVIEWER = {Mee\ Seong\ Im},
       DOI = {10.1215/00127094-2020-0026},
       URL = {https://doi.org/10.1215/00127094-2020-0026},
}

@article {Axtell,
    AUTHOR = {Jonathan Axtell},
     TITLE = {Spin polynomial functors and representations of {S}chur
              superalgebras},
   JOURNAL = {Represent. Theory},
  FJOURNAL = {Representation Theory. An Electronic Journal of the American
              Mathematical Society},
    VOLUME = {17},
      YEAR = {2013},
     PAGES = {584--609},
      ISSN = {1088-4165},
   MRCLASS = {18D20 (14L15 16D90 17A70 20Gxx)},
  MRNUMBER = {3138585},
MRREVIEWER = {Andrey\ Yu.\ Lazarev},
       DOI = {10.1090/S1088-4165-2013-00445-4},
       URL = {https://doi.org/10.1090/S1088-4165-2013-00445-4},
}

@article {Giordano,
    AUTHOR = {Iacopo Giordano},
     TITLE = {Additive polynomial superfunctors and cohomology},
   JOURNAL = {J. Algebra},
  FJOURNAL = {Journal of Algebra},
    VOLUME = {632},
      YEAR = {2023},
     PAGES = {462--486},
      ISSN = {0021-8693,1090-266X},
   MRCLASS = {18G40},
  MRNUMBER = {4603824},
       DOI = {10.1016/j.jalgebra.2023.05.041},
       URL = {https://doi.org/10.1016/j.jalgebra.2023.05.041},
}

@article {Donkin,
    AUTHOR = {Stephen Donkin},
     TITLE = {Symmetric and exterior powers, linear source modules and
              representations of {S}chur superalgebras},
   JOURNAL = {Proc. London Math. Soc. (3)},
  FJOURNAL = {Proceedings of the London Mathematical Society. Third Series},
    VOLUME = {83},
      YEAR = {2001},
    NUMBER = {3},
     PAGES = {647--680},
      ISSN = {0024-6115,1460-244X},
   MRCLASS = {20G05 (16S99 16W55 20C30)},
  MRNUMBER = {1851086},
MRREVIEWER = {Karin\ Erdmann},
       DOI = {10.1112/plms/83.3.647},
       URL = {https://doi.org/10.1112/plms/83.3.647},
}

@book {Alperin,
    AUTHOR = {Alperin, J. L.},
     TITLE = {Local representation theory},
    SERIES = {Cambridge Studies in Advanced Mathematics},
    VOLUME = {11},
      NOTE = {Modular representations as an introduction to the local
              representation theory of finite groups},
 PUBLISHER = {Cambridge University Press, Cambridge},
      YEAR = {1986},
     PAGES = {x+178},
      ISBN = {0-521-30660-4; 0-521-44926-X},
   MRCLASS = {20-02 (20C20)},
  MRNUMBER = {860771},
MRREVIEWER = {Peter\ W.\ Donovan},
       DOI = {10.1017/CBO9780511623592},
       URL = {https://doi.org/10.1017/CBO9780511623592},
}

@article {Brundan-Kujawa,
    AUTHOR = {Jonathan Brundan and Jonathan Kujawa},
     TITLE = {A new proof of the {M}ullineux conjecture},
   JOURNAL = {J. Algebraic Combin.},
  FJOURNAL = {Journal of Algebraic Combinatorics. An International Journal},
    VOLUME = {18},
      YEAR = {2003},
    NUMBER = {1},
     PAGES = {13--39},
      ISSN = {0925-9899,1572-9192},
   MRCLASS = {20C30 (05E05 05E10 05E15 20G05)},
  MRNUMBER = {2002217},
MRREVIEWER = {Christine\ Bessenrodt},
       DOI = {10.1023/A:1025113308552},
       URL = {https://doi.org/10.1023/A:1025113308552},
}

@misc{benson2025groupschemesliealgebras,
      title={Group schemes and their Lie algebras over a symmetric tensor category}, 
      author={Dave Benson and Julia Pevtsova},
      year={2025},
      eprint={2507.02031},
      archivePrefix={arXiv},
      primaryClass={math.RT},
      url={https://arxiv.org/abs/2507.02031}, 
}
\end{document}